\documentclass[11pt]{article}
\makeatother

\normalsize
\usepackage{multirow}
\usepackage{makecell} 
\usepackage{color}
\usepackage{hyperref} 
\usepackage{graphicx} 
\usepackage{epstopdf} 
\usepackage{amsfonts}
\usepackage{amsmath}
\usepackage{amssymb}
\usepackage{epstopdf}
\usepackage{amsthm}
\usepackage{epsfig}
\usepackage[mathscr]{euscript}
\usepackage[numbers,sort&compress]{natbib}
\newtheorem{remark}{Remark}[section]
\newtheorem{Example}{Example}

\newtheorem{theorem}{Theorem}[section]
\newtheorem{lemma}{Lemma}[section]
\newtheorem{col}{Corollary}[section]

\newcommand{\be}{\begin{equation}}
\newcommand{\ee}{\end{equation}}

\begin{document}
\setlength{\baselineskip}{16pt} \pagestyle{myheadings}

\topmargin0mm \topsep 2mm \markboth{\centerline{\small
}}{\centerline{\small\  }}\thispagestyle{empty}

\title{\bf Dynamics of a memory-based diffusion model with spatial heterogeneity and nonlinear boundary condition}
\author{\small Quanli Ji$^{a,}{^b}$, \ \ Ranchao Wu$^a$\footnote{Corresponding author. E-mail address: rcwu@ahu.edu.cn (R. C. Wu)}, \ \
Tonghua Zhang${^b}$\\
\centerline{\footnotesize$^a$ School of Mathematical Sciences,  Anhui University, Hefei 230601, China}\\
\centerline{\footnotesize$^b$ Department of Mathematics, Swinburne University of Technology, Hawthorn, VIC 3122, Australia}}

\date{}

\maketitle

\begin{quote}
\noindent {\bf Abstract.}
In this work, we study the dynamics of a spatially heterogeneous single population model with the memory effect and nonlinear boundary condition.
By virtue of the implicit function theorem and Lyapunov-Schmidt reduction, spatially nonconstant positive steady state solutions appear from two trivial solutions,
respectively. By using bifurcation analysis, the Hopf bifurcation associated with one spatially nonconstant positive steady state is found to occur.
The results complement the existing ones. Specifically, it is found  that with the interaction of spatial heterogeneity and nonlinear boundary condition, when the memory
term is stronger than the interaction of the interior reaction term and the boundary one,  the memory-based diffusive model has a single stability switch from stability to
instability, with the increase of the delayed memory value. Therefore, the memory delay  will lead to a single stability  switch of such memory-based diffusive model  and
consequently the Hopf bifurcation will happen in the model.


\end{quote}

\normalsize


\section{Introduction}

It is well known that various reaction-diffusion models based on the Fick's law have been extensively proposed to understand practical problems of many disciplines,
such as ecology, chemistry and biology \cite{Ecology-2001,Turing-1952,Murray-book}. Aside from the Fick's law, some practical problems with the biased animal movements
are affected by other factors, including heterogeneous living environment, the resource distribution, the spatial memory and cognition \cite{2003-Cantrell-Cosner-book,morphogen,2013-Ecology-letter-WF.Fagan}. In view of some biological taxis phenomena, many chemotaxis reaction-diffusion models have
been widely established to describe the attractive or repulsive animal movement, which could interpret the biased phenomena \cite{2010prey-taxis,2011-Physica-D-Chemotaxis,2014-Physica-KS-Chemotaxis}. Recently, it has been emphasized in a review paper \cite{2013-Ecology-letter-WF.Fagan}
that the spatial memory and cognition are of the significance to describe the biased animal movement, especially movements of high-developed animals.
According to the inspiration of the review paper \cite{2013-Ecology-letter-WF.Fagan}, Shi et al. \cite{2020-JDDE-Shi-memory} first incorporated the episodic-like spatial memory of animals into
the reaction-diffusion model via a modified Fick's law
and established the following memory-based diffusion model in the self-contained way
\begin{equation}\label{eq1}
 \begin{cases}
    \frac{\partial{u}}{\partial{t}}=d_{1}\triangle{u}+d_{2}\nabla\cdot({u}\nabla{u_{\sigma}})+g({u}),  &\text{$x \in \Omega, t >0 $},\\
    \partial_{\overrightarrow{n}}u=0,                           &\text{$x \in \partial{\Omega}, t >0 $},
  \end{cases}
\end{equation}
where $u=u(x,t)$ describes the population density at location $x$ and time $t$, and $u_{\sigma}=u(x,t-\sigma)$, the time delay $\sigma\geq{0}$ represents the averaged memory period,
and it was shown in \cite{2020-JDDE-Shi-memory} that the stability of a spatially homogeneous steady state fully depends on the relationship between two coefficients $(d_{1}$ and
$d_{2})$ but is independent of time delay. In the sequel, Shi et al. \cite{2020-Nonlinearity-shijunping-two-delay} considered the factor of maturation delay into memory-based
diffusion model (\ref{eq1}), for which spatially homogeneous and nonhomogeneous periodic solutions could be obtained from the Hopf bifurcation respectively.
Song et al. \cite{2020-JDE-songyongli-TH with memory} investigated the memory-based diffusion model (\ref{eq1}) with nonlocal effect in reaction term, which
showed that the emergence of Turing-Hopf and double Hopf bifurcations could be induced by the joint effects of spatial memory and nonlocal reaction.
Then Song et al. \cite{2021-SAM-consumer-memory} proposed a novel consumer-resource model with random diffusions and consumer's spatial memory, in which the resource species has no memory or cognition, and it is found that Hopf bifurcation and stability switches emerge via considering the memory-based diffusion coefficient and the
averaged memory period of the consumer as the key parameters.
Shi et al. \cite{2021JMB-distributed-memory} introduced the spatiotemporally distributed memory into the model (\ref{eq1}), which can lead to rich spatiotemporal dynamics through spatially nonhomogeneous steady states and Hopf bifurcations.
Recently Shi and Shi \cite{2024JDE-distributed} replaced the Nuemann boundary condition and the discrete memory of the model (\ref{eq1}) with the logistic reaction term into the Dirichlet boundary condition and temporally distributed memory, and for the strong kernel case, it is shown that Hopf bifurcation and stability switches occur in the intermediate value of delays.
For more works, refer scholars to \cite{2021JDE-S,2023JDE-LS,2023JMB-ZH,2023JMBopen-problems} and references therein.

In recent years, the combination of animal movements and spatial heterogeneity has been one of the dominant topics on the population dynamics
\cite{2003-Cantrell-Cosner-book,2019JMPA-heterogeneity,2002-Lou-heterogeneity,2006-JDE-Louyuan-h,2013-JDE-HeXq,2016-CPAM-HeXq,2018-JDE-chen-lou-wei,2019-DCDS-Shi-heterogeneity,
2021-JDDE-wangchuncheng-memory,2023AML,2023DCDSB}.
In addition, heterogeneous resource supply is more realistic as it is inevitable that there exist some factors, such as humidity, topography, climate, etc., which could restrict the resource
distribution \cite{2003-Cantrell-Cosner-book,2019JMPA-heterogeneity,2002-Lou-heterogeneity,2006-JDE-Louyuan-h}.
Therefore, Wang et al. \cite{2021-JDDE-wangchuncheng-memory} incorporated the spatially heterogeneous resource function into the memory-based diffusion model
\begin{equation}\label{eq2}
 \begin{cases}
    \frac{\partial{u}}{\partial{t}}=d_{1}\triangle{u}+d_{2}\nabla\cdot({u}\nabla{u_{\sigma}})+u(m(x)-u),  &\text{$x \in \Omega, t >0 $},\\
    \partial_{\overrightarrow{n}}u=0,                           &\text{$x \in \partial{\Omega}, t >0 $},
  \end{cases}
\end{equation}
where $m(x)$ is the local carrying capacity or the intrinsic growth rate that represents the situation of the resources.
Without memory effect, $i.e.$, $d_{2}=0$, the model (\ref{eq2}) becomes  the classic logistic reaction-diffusion model with spatial heterogeneity    \cite{2003-Cantrell-Cosner-book,2006-JDE-Louyuan-h}, which showed that there is a unique positive spatially nonhomogeneous steady state solution and it is globally asymptotically stable. 
However, with the memory effect, $i.e.$, $d_{2}\neq0$, it was shown in \cite{2021-JDDE-wangchuncheng-memory} that as memory-induced diffusion rate is relatively small, a spatially nonhomogeneous steady state solution of global existence is linearly stable;
while spatially nonhomogeneous periodic solutions could emerge via the Hopf bifurcation as memory-induced diffusion rate exceeds a critical value.
For more details, please see \cite{2003-Cantrell-Cosner-book,2019JMPA-heterogeneity,2002-Lou-heterogeneity,2006-JDE-Louyuan-h,2013-JDE-HeXq,2016-CPAM-HeXq,2018-JDE-chen-lou-wei,2019-DCDS-Shi-heterogeneity,
2021-JDDE-wangchuncheng-memory,2023AML,2023DCDSB,2020-DCDSA-anqi-memory-nonlocal,1996busenberg,2015-JDE-guoshangjiang} and references therein.

It was frequently found that once the individual animals reach the boundary $\partial{\Omega}$ of the habitat, they will be taken outside the habitat at a rate 
which also depends on its population density, see \cite{2006NBC-CR} for the different  boundary condition. Interested readers could refer to the literatures \cite{2003-Cantrell-Cosner-book,2020-JDDE-Shi-memory,2021-JDDE-wangchuncheng-memory,2006NBC-CR,2021JDE-guo} and references therein,
in order to better understand the biological justification of the nonlinear boundary condition. 
What about the dynamics in the heterogeneous reaction-diffusion model with the interaction of memory effect and nonlinear boundary condition? That is interesting. 
With the incorporation of nonlinear boundary condition into the model, it takes the following form
\begin{equation}\label{eq3}
 \begin{cases}
    \frac{\partial{u}}{\partial{t}}=\triangle{u}+d\nabla\cdot({u}\nabla{u_{\sigma}})+{\lambda}uf(x,u),  &\text{$x \in \Omega, t >0 $},\\
     \partial_{\overrightarrow{n}}u={\lambda}r(x)g(u),                           &\text{$x \in \partial{\Omega}, t >0 $},
  \end{cases}
\end{equation}
and the steady state solutions of (\ref{eq3}) satisfy
\begin{equation}\label{eq4}
 \begin{cases}
    0=\triangle{u}+d\nabla\cdot({u}\nabla{u})+{\lambda}uf(x,u),  &\text{$x \in \Omega$},\\
     \partial_{\overrightarrow{n}}u={\lambda}r(x)g(u),                           &\text{$x \in \partial{\Omega}$},
  \end{cases}
\end{equation}
where $u=u(x,t)$ is the population density at the spatial location $x$ and at time $t$, $u_{\sigma}=u(x,t-\sigma)$, one delay $\sigma\geq{0}$ is the averaged memory period, $d$ describes the coefficient of memory-based diffusion, $\Omega$ is a bounded domain in $\mathbb{R}^{N}$~$(1\leq{N}\leq3)$ with a smooth boundary $\partial{\Omega}$, $\partial_{\overrightarrow{n}}u=\nabla{u}\cdot{\overrightarrow{n}}$, $\overrightarrow{n}$ is the outward unit normal vector on $\partial{\Omega}$, $\lambda$ is a positive parameter.

In the past decades, these problems of reaction-diffusion models with nonlinear boundary conditions have been attracting increasing interests. To the best of our knowledge,
there have been some results about reaction-diffusion models with nonlinear boundary conditions. The blowup of solutions were studied in \cite{1974NBC-HA,1975NBC-WW,1991NBC-JL,1993NBC-MX,1994NBC-HB}; the wellposedness and asymptotical behavior of solutions were investigated in \cite{1997NBC-AN,1999NBC-JM,2000NBC-JM,2001NBC-AR,2002NBC-AR,2007CVPAO,2016NBC-AR};
and the boundary layer solutions were considered in \cite{2004NBC-JM,2005NBC-XC,2008NBC-LD,2011NBC-MM};
the existence, uniqueness and stability of steady state solutions were analysed by bifurcation method and other related methods
in \cite{2000NBC-UM,2002NBC-UM,2006NBC-CR,2007NBC-JG,2009NBC-CR,2011NBC-GM,2012NBC-UM,2013NBC-UM,2014NBC-HQ,2015NBC-HQ,2016NBC-HQ,2018JDE-shi,2023JDE,
2021JDE-guo,2023JDE-guo,2023Proc-guo}.
Recently, Liu and Shi \cite{2018JDE-shi} considered a scalar reaction-diffusion model with nonlinear boundary condition,  and the existence and stability
of the bifurcated steady states were obtained through a unified approach. Inspired by the work \cite{2018JDE-shi}, Li et al. \cite{2023JDE} proposed and investigated a
reaction-diffusion-advection model with nonlinear boundary condition, for which the existence and stability of the bifurcated steady states  were analysed.
Guo \cite{2021JDE-guo} established a single population reaction-diffusion model with nonlocal delayed and nonlinear boundary condition
and studied the existence, stability, and multiplicity of the steady states and periodic solutions via the implicit function theorem and Lyapunov-Schmidt reduction.
Note that global dynamics of a Lotka-Volterra competition-diffusion system with nonlinear boundary conditions were investigated in \cite{2023JDE-guo}.
Please see \cite{1974NBC-HA,1975NBC-WW,1991NBC-JL,1993NBC-MX,1994NBC-HB,
1997NBC-AN,1999NBC-JM,2000NBC-JM,2001NBC-AR,2002NBC-AR,2007CVPAO,2016NBC-AR,
2004NBC-JM,2005NBC-XC,2008NBC-LD,2011NBC-MM,
2000NBC-UM,2002NBC-UM,2006NBC-CR,2007NBC-JG,2009NBC-CR,2011NBC-GM,2012NBC-UM,2013NBC-UM,2014NBC-HQ,2015NBC-HQ,2016NBC-HQ,2018JDE-shi,2023JDE,
2021JDE-guo,2023JDE-guo,2023Proc-guo} and references therein for more details.

Throughout the paper, we always assume that $f$, $g$, $r$ satisfy the following basic assumption:
\begin{eqnarray*}
(A0) ~~
\begin{cases}
\text{$f\in{C^{1,\tilde{\alpha}}(\bar{\Omega}\times\mathbb{R},\mathbb{R})}$, $g\in{C^{1,\tilde{\alpha}}(\mathbb{R},\mathbb{R})}$, $g(0)=0$, $r\in{C^{1,\tilde{\alpha}}(\partial{\Omega},\mathbb{R})}$ for $\tilde{\alpha}\in(0,1)$}. ~~~~~~~~~~~~~~~\\
\text{$f(x,\cdot)\in{L^{\infty}}(\Omega)$, $f_{u}(x,\cdot)\in{L^{\infty}}(\Omega)$, $g(\cdot)\in{L^{\infty}}$, $g_{u}(\cdot)\in{L^{\infty}}$,
$r(x)\in{L^{\infty}}(\partial\Omega)$, $r(x)g(u)\leq{0}$}.
\end{cases}
\end{eqnarray*}
We would like to investigate the existence and stability of  nonconstant steady states and whether there will exist the Hopf bifurcation as the memory delay varies.
It is found that the memory delay $\sigma$ could induce  the occurrence of  the Hopf bifurcation in model (\ref{eq3}) near $u_{\lambda}^{*}$
if the memory reaction term is stronger than the interaction of the interior reaction term and the boundary reaction one. The result is different from those in 
the previous literatures. 

The layout of the paper is organized as follows. In Section 2, we introduce the existence of the bifurcated nontrivial steady state solutions.
In Section 3, stability and bifurcation analysis associated with one nonconstant positive steady state are carried out.
In Section 4, the stability analysis associated with the other nonconstant positive steady state is presented.
Finally, some discussions are drawn in Section 5.

Throughout the paper, we shall introduce some notations for later use. Denote by $L^{p}(\Omega)(p\geq N)$
the Lebesgue space of integrable functions defined on $\Omega$, and let $W^{k,p}(k\geq0)$ be the Sobolev space
of the $L^{p}-$functions $f(x)$ defined on $\Omega$ whose derivatives $\frac{d^{n}}{dx^{n}}f(n=1,...,k)$ also belong to $L^{p}(\Omega)$.
Define spaces $\mathcal{X}:=W^{2,p}(\Omega)$ and $\mathcal{Y}:=L^{p}(\Omega)\times W^{1-\frac{1}{p},p}(\partial{\Omega})$
where $p>N$. For a space $Z$, define by $Z_{\mathbb{C}}:=Z\oplus{i}Z=\{x_{1}+ix_{2}|x_{1},x_{2}\in Z\}$.
For a linear operator L, we respectively use $\mathcal{N}(L)$ and $\mathcal{R}(L)$ as the null space and the range space,
and denote by $L[w]$ the image of $w$ under the linear mapping $L$. For a multilinear operator $L$, denote by $L[w_{1},w_{2},...,w_{j}]$
the image of $(w_{1},w_{2},...,w_{j})$ under $L$. For a nonlinear operator $T$, denote by $D_{u}T$ the partial derivative of $T$ with respect to $u$.
Denote by $\mathbb{N}$ the set of positive integers and $\mathbb{N}_{0}=\mathbb{N}\bigcup{\{0\}}$ the set of nonnegative integers.

\section{Existence of nontrivial steady states}
This section is devoted to the existence of spatially nonhomogeneous positive steady state of the model (\ref{eq3}),
which is essentially the steady state solutions of (\ref{eq4}). In what follows, we consider the nontrivial solutions of (\ref{eq4}) bifurcating from the known trivial solutions.
To this end, define a nonlinear mapping $T: \mathcal{X}\times{\mathbb{R}}\rightarrow{\mathcal{Y}}$ by
\begin{equation}\label{eq5}
T(u,\lambda)=\bigg(\triangle{u}+d\nabla\cdot({u}\nabla{u})+{\lambda}uf(x,u),
\partial_{\overrightarrow{n}}u-{\lambda}r(x)g(u)\bigg)^{T}
:=\left(\begin{array}{c}
   T_{1}(u,\lambda)\\
   T_{2}(u,\lambda)
\end{array}\right),
\end{equation}
and it follows from the assumption $(A0)$ that any solution $(\lambda,u)$ of $T(u,\lambda)=0$ is indeed a classical solution of class $C^{2,\tilde{\alpha}}(\bar{\Omega})$.
For the sake of completeness, we first give trivial solutions. It is easy to see that (\ref{eq4}) has a line of trivial solutions
\begin{align}\label{eq6}
\Gamma_{0}=\big\{(0,\lambda):\lambda>0\big\},
\end{align}
in the $u-\lambda$ plane. On the other hand, any constant $u_{1}\geq0$ is a solution to (\ref{eq4}) for $\lambda=0$.
Therefore, (\ref{eq4}) has another line of trivial solutions
\begin{align}\label{eq7}
\Gamma_{u_{1}}=\big\{(u_{1},0): u_{1}\in\mathbb{R_{+}}\big\},
\end{align}
in the $u-\lambda$ plane. In the following, it is shown that the nontrivial solutions of (\ref{eq4}) can emerge from $\Gamma_{0}$
at some bifurcation point $(u,\lambda)=(0,\lambda_{*})$, or from $\Gamma_{u_{1}}$
at some bifurcation point $(u,\lambda)=(u_{1},0)$. In a special case, the nontrivial solutions also emerge from $(u,\lambda)=(0,0)$,
the intersection point of $\Gamma_{0}$ and $\Gamma_{u_{1}}$.

A point $(0,\lambda_{*})$ is called a bifurcation point on the line of $\Gamma_{0}$
if there exists a sequence $(u^{n},\lambda^{n})$ of solutions to (\ref{eq4}) such that
$0\neq u^{n}\rightarrow{0}$, $\lambda^{n}\rightarrow{\lambda_{*}}$ in $C(\bar{\Omega})$ as $n\rightarrow{\infty}$.
The bifurcation point on the line of $\Gamma_{u_{1}}$ can be defined similarly.
In the sequel, we study the bifurcation of nontrivial solutions of (\ref{eq4})
from the two lines of trivial solutions (\ref{eq6}) and (\ref{eq7}).
We first obtain the following lemma with respect to the possible bifurcation points
by the arguments similar to those in \cite{2018JDE-shi}.

\begin{lemma}\label{le1}
Suppose that $(A0)$ is satisfied.
\begin{itemize}
\item[(i)] If $(0,\lambda_{*})$ with $\lambda_{*}>0$ is a bifurcation point of (\ref{eq4})
regarding the trivial branch $\Gamma_{0}$, then $\lambda_{*}>0$ is an eigenvalue of
\begin{equation}\label{eq8}
 \begin{cases}
    -\triangle{v}={\lambda}f(x,0)v,               &\text{$x \in \Omega$},\\
     \partial_{\overrightarrow{n}}v={\lambda}r(x)g_{u}(0)v,           &\text{$x \in \partial{\Omega}$}.
  \end{cases}
\end{equation}

\item[(ii)] If $(u_{1},0)$ with $u_{1}>0$ is a bifurcation point of (\ref{eq4})
regarding the trivial branch $\Gamma_{u_{1}}$, then $u_{1}>0$ satisfies
\begin{eqnarray*}
(A1) ~~u_{1}\int_{\Omega}f(x,u_{1})dx+(1+du_{1})g(u_{1})\int_{\partial{\Omega}}r(x)dS=0.~~~~~~~~~~~~~~~~~~~~~~~~~~~~~~~~~~~~~~~~~~~~~~~~~
\end{eqnarray*}

\end{itemize}
\end{lemma}
\begin{proof}
Define the partial derivative of $T(u,\lambda)$ regarding $u$ evaluated at $(u_{*},\lambda)$, which is given by
\begin{equation}\label{eq9}
\mathcal{L}_{\lambda}v:=D_{u}T(u_{*},\lambda)[v]=
\left(\begin{array}{c}
   (1+du_{*})\triangle{v}+{\lambda}u_{*}f_{u}(x,u_{*})v+{\lambda}f(x,u_{*})v\\
   \partial_{\overrightarrow{n}}v-{\lambda}r(x)g_{u}(u_{*})v
\end{array}\right).
\end{equation}

$(i)$    If the conclusion is not true, then $\lambda_{*}$ is not an eigenvalue of (\ref{eq8}),
further $D_{u}T(0,\lambda_{*})$ is a homeomorphism for all $\lambda$ near $\lambda_{*}$.
It follows from the implicit function theorem that the trivial solution  $(0,\lambda)$ is
the unique solution of $T(u,\lambda)$ near $(0,\lambda_{*})$. Therefore, $\lambda_{*}$ is not a bifurcation point
along the line $\Gamma_{0}$. Hence if $\lambda_{*}$ is a bifurcation point,
then $\lambda_{*}$ must be an eigenvalue of (\ref{eq8}).

$(ii)$   If $(u_{1},0)$ with $u_{1}>0$ is a bifurcation point of (\ref{eq4})
regarding the trivial branch $\Gamma_{u_{1}}$, then it follows from the definition that there exists a sequence
$(u^{n},\lambda^{n})$ of solutions to (\ref{eq4}) with
\begin{eqnarray*}
0\neq \lambda^{n}\rightarrow{0} ~\text{and $\big\|{u^{n}-u_{1}}\big\|_{\mathcal{X}}\rightarrow{0}$, when $n\rightarrow\infty$}.
\end{eqnarray*}

By integration and the formula of Green, we have
\begin{eqnarray}\label{eq10}
\begin{split}
\lambda^{n}\int_{\Omega}u^{n}f(x,u^{n})dx
=&-\int_{\Omega}\triangle{u^{n}}dx-d\int_{\Omega}\nabla\cdot({u^{n}}\nabla{u^{n}})dx\\
=&-\int_{\partial{\Omega}}\frac{\partial{u^{n}}}{\partial{\overrightarrow{n}}}dS-d\int_{\partial{\Omega}}u^{n}\frac{\partial{u^{n}}}{\partial{\overrightarrow{n}}}dS\\
=&-\lambda^{n}\int_{\partial{\Omega}}(1+du^{n})r(x)g(u^{n})dS,
\end{split}
\end{eqnarray}
which implies that $\int_{\Omega}u^{n}f(x,u^{n})dx+\int_{\partial{\Omega}}(1+du^{n})r(x)g(u^{n})dS=0$ since $\lambda^{n}\neq{0}$.
By taking $n\rightarrow\infty$, we obtain $(A1)$.
\end{proof}

\subsection{Local bifurcation from $\Gamma_{0}$}

~\quad In the subsection, we investigate the bifurcation on $\Gamma_{0}=\big\{(0,\lambda):\lambda>0\big\}$.
Lemma \ref{le1} implies that $\lambda_{*}$ must be an eigenvalue of (\ref{eq8})
if $(0,\lambda_{*})$ with $\lambda_{*}>0$ is a bifurcation point.
Regarding the existence of such a principal eigenvalue, we have the following result from Umezu \cite{2006Umezu}.

\begin{lemma}\label{le2}
Consider the eigenvalue problem
\begin{equation}\label{eq11}
 \begin{cases}
    -\triangle{u}={\lambda}f(x,0)u,                     &\text{$x \in \Omega$},\\
     \partial_{\overrightarrow{n}}u={\lambda}r(x)g_{u}(0)u,    &\text{$x \in \partial{\Omega}$}.
  \end{cases}
\end{equation}
Assume that
\begin{equation}\label{eq12}
\text{either $f(x,0)\nleq{0}$ in $\Omega$ or $r(x)g_{u}(0)\nleq{0}$ on $\partial{\Omega}$},
\end{equation}
the problem (\ref{eq11}) has a unique positive principal eigenvalue $\lambda_{1}$ if and only if
\begin{equation}\label{eq13}
\text{$\int_{\Omega}f(x,0)dx+g_{u}(0)\int_{\partial{\Omega}}r(x)dS<0$},
\end{equation}
and it is characterized by the formula
\begin{eqnarray*}
\lambda_{1}=\inf\bigg\{Qu:u\in W^{1,2}(\Omega), \int_{\Omega}f(x,0)u^{2}dx+g_{u}(0)\int_{\partial{\Omega}}r(x)u^{2}dS>0\bigg\},
\end{eqnarray*}
where
\begin{eqnarray*}
Qu=\frac{\int_{\Omega}{\big|\nabla{u}\big|}^{2}dx}{\int_{\Omega}f(x,0)u^{2}dx+g_{u}(0)\int_{\partial{\Omega}}r(x)u^{2}dS}, ~~~~\text{for $u\in{W^{1,2}(\Omega)}$}.
\end{eqnarray*}
\end{lemma}

Under the assumptions (\ref{eq12}) and (\ref{eq13}), it follows from Lemma \ref{le2} that $\lambda_{1}$ is the principal eigenvalue of (\ref{eq8}),
with an associated eigenfunction $\varphi_{1}$ satisfying $\int_{\Omega}\varphi_{1}^{2}(x)dx=1$.
Hence, Lemma \ref{le1} and Lemma \ref{le2} imply that $\mathcal{N}(\mathcal{L}_{\lambda_{1}})=\text{span}\{\varphi_{1}\}$.

It is easy to see that codim$\mathcal{R}(\mathcal{L}_{\lambda_{1}})=1$.
In fact, let $(y_{1},y_{2})^{T}\in\mathcal{R}(\mathcal{L}_{\lambda_{1}})$ and $v\in\mathcal{X}$ satisfy
\begin{equation}\label{eq14}
 \begin{cases}
    \triangle{v}+{\lambda_{1}}f(x,0)v=y_{1},                             &\text{$x \in \Omega$},\\
     \partial_{\overrightarrow{n}}v-{\lambda_{1}}r(x)g_{u}(0)v=y_{2},    &\text{$x \in \partial{\Omega}$},
  \end{cases}
\end{equation}
and for $v=\varphi_{1}$ in particular, it follows from Lemma \ref{le1} and Lemma \ref{le2} that we have
\begin{equation}\label{eq15}
 \begin{cases}
    -\triangle{\varphi_{1}}={\lambda_{1}}f(x,0)\varphi_{1},                            &\text{$x \in \Omega$},\\
     \partial_{\overrightarrow{n}}\varphi_{1}={\lambda_{1}}r(x)g_{u}(0)\varphi_{1},    &\text{$x \in \partial{\Omega}$}.
  \end{cases}
\end{equation}

Multiplying the equations in (\ref{eq14}) and (\ref{eq15}) by $\varphi_{1}$ and $v$, respectively, and integrating over the results, we obtain
\begin{eqnarray*}
\begin{split}
\int_{\Omega}\varphi_{1}y_{1}dx
=&\int_{\Omega}\varphi_{1}\triangle{v}dx+\lambda_{1}\int_{\Omega}f(x,0)\varphi_{1}vdx=\int_{\Omega}\varphi_{1}\triangle{v}dx-\int_{\Omega}v\triangle{\varphi_{1}}dx\\
=&\int_{\partial{\Omega}}\varphi_{1}\frac{\partial{v}}{\partial{\overrightarrow{n}}}dS
-\int_{\partial{\Omega}}v\frac{\partial{\varphi_{1}}}{\partial{\overrightarrow{n}}}dS
=\int_{\partial{\Omega}}\varphi_{1}y_{2}dS.
\end{split}
\end{eqnarray*}

Therefore, it is shown that $(y_{1},y_{2})^{T}\in\mathcal{R}(\mathcal{L}_{\lambda_{1}})$ if and only if
\begin{eqnarray*}
\int_{\Omega}y_{1}(x)\varphi_{1}(x)dx-\int_{\partial{\Omega}}y_{2}(x)\varphi_{1}(x)dS=0.
\end{eqnarray*}

Hence codim$\mathcal{R}(\mathcal{L}_{\lambda_{1}})=1$ and $\mathcal{L}_{\lambda_{1}}$ is a Fredholm operator with index zero.
Thus, we decompose the spaces: $\mathcal{X}=\mathcal{N}(\mathcal{L}_{\lambda_{1}})\oplus{\mathcal{X}_{1}}$ and  $\mathcal{Y}=\mathcal{R}(\mathcal{L}_{\lambda_{1}})\oplus{\mathcal{Y}_{1}}$.
In the sequel, we perform a Lyapunov-Schmidt reduction. Define projections
\begin{eqnarray*}
Q_{\mathcal{X}}:\mathcal{X}\rightarrow{\mathcal{X}_{1}},  \qquad\qquad   Q_{\mathcal{Y}}:\mathcal{Y}\rightarrow{\mathcal{Y}_{1}},
\end{eqnarray*}
then $T(u,\lambda)=0$ is equivalent to
\begin{eqnarray}\label{eq16}
u=u_{0}+u_{2},  \qquad   Q_{\mathcal{Y}}T(u_{0}+u_{2},\lambda)=0,  \qquad   (I-Q_{\mathcal{Y}})T(u_{0}+u_{2},\lambda)=0,
\end{eqnarray}
where
\begin{eqnarray*}
u_{0}=(I-Q_{\mathcal{X}})u\in\mathcal{N}(\mathcal{L}_{\lambda_{1}}),  \qquad   u_{2}=Q_{\mathcal{X}}u\in{\mathcal{X}_{1}}.
\end{eqnarray*}

Notice that $(I-Q_{\mathcal{Y}})T(0,\lambda_{1})=0$ and $(I-Q_{\mathcal{Y}})T_{u_{0}}(0,\lambda_{1})=\mathcal{L}_{\lambda_{1}}$.
Then it follows from the implicit function theorem that we have a continuously differentiable map
$h:{\mathcal{U}}\rightarrow{\mathcal{X}_{1}}$ such that $h(0,\lambda)=0$ and
\begin{eqnarray}\label{eq17}
(I-Q_{\mathcal{Y}})T(u_{0}+h(u_{0},\lambda),\lambda)\equiv{0},
\end{eqnarray}
where ${\mathcal{U}}$ is an open neighborhood of $(0,\lambda_{1})$ in $\mathcal{N}(\mathcal{L}_{\lambda_{1}})\times{\mathbb{R}}$.
Substituting $u_{2}=h(u_{0},\lambda)$ into the second equation of (\ref{eq16}) gives
\begin{eqnarray}\label{eq18}
\mathcal{T}(u_{0},\lambda):=Q_{\mathcal{Y}}T(u_{0}+h(u_{0},\lambda),\lambda)=0.
\mathcal{}\end{eqnarray}
Obviously, $\mathcal{T}(0,\lambda)=0$ and $\mathcal{T}_{u_{0}}(0,\lambda_{1})=0$.
Thus, each solution to $\mathcal{T}(u_{0},\lambda)=0$ in ${\mathcal{U}}$ one-to-one corresponds to some solution to $T(u,\lambda)=0$.

Now we give an explicit representation of the projection $Q_{\mathcal{Y}}$. Recall that if $\text{dim$\mathcal{Y}_{1}$}=1$,
then there exists $\phi\in{\mathcal{Y}}$ satisfying $\big\|{\phi}\big\|_{\mathcal{Y}}=1$ such that ${\mathcal{Y}_{1}}=\text{span}\{\phi\}$.
This means that there exists a constant $z$ such that $Q_{\mathcal{Y}}y=z\phi$ for $y\in\mathcal{Y}$. By the Hahn-Banach Theorem \cite{1980-Yoshida}, there exists a vector $\Psi$ in the dual space $\mathcal{Y}^{*}$ of $\mathcal{Y}$ such that
$\langle{\Psi,\phi}\rangle_{1}=1$ and $\langle{\Psi,y}\rangle_{1}=0$ for all $y\in{\mathcal{R}(\mathcal{L}_{\lambda_{1}})}$,
where $\langle{\cdot,\cdot}\rangle_{1}: \mathcal{Y}^{*}\times\mathcal{Y}\rightarrow{\mathbb{R}}$ denotes the duality between $\mathcal{Y}^{*}$ and $\mathcal{Y}$
and is defined as
\begin{eqnarray*}
\left \langle l,y \right \rangle_{1}=\int_{\Omega}l(x)y_{1}(x)dx-\int_{\partial{\Omega}}l(x)y_{2}(x)dS,
~~\text{for all $l\in\mathcal{Y}^{*}$ and $y=(y_{1},y_{2})^{T}\in\mathcal{Y}$}.
\end{eqnarray*}
Obviously, there exists a vector $\Psi\in\mathcal{Y}^{*}$ such that
\begin{eqnarray*}
\left \langle \Psi,y \right \rangle_{1}=\int_{\Omega}\varphi_{1}(x)y_{1}(x)dx-\int_{\partial{\Omega}}\varphi_{1}(x)y_{2}(x)dS,
\end{eqnarray*}
and hence that $\mathcal{N}(\Psi)=\mathcal{R}(\mathcal{L}_{\lambda_{1}})$. Further, we have $\langle{\Psi,y}\rangle_{1}=\langle{\Psi,Q_{\mathcal{Y}}y+(I-Q_{\mathcal{Y}})y}\rangle_{1}
=\langle{\Psi,Q_{\mathcal{Y}}y}\rangle_{1}+\langle{\Psi,(I-Q_{\mathcal{Y}})y}\rangle_{1}
=\langle{\Psi,Q_{\mathcal{Y}}y}\rangle_{1}=\langle{\Psi,z\phi}\rangle_{1}=z$, which means that the projection $Q_{\mathcal{Y}}$ is given by
$Q_{\mathcal{Y}}y=\langle{\Psi,y}\rangle_{1}\phi$ for any $y\in\mathcal{Y}$.

For every $u_{0}=\vartheta\varphi_{1}\in\mathcal{N}(\mathcal{L}_{\lambda_{1}})$ with $\vartheta\in{\mathbb{R}}$,
substituting this into (\ref{eq18}) and calculating the inner product with $\Psi$, we obtain $G(\vartheta,\lambda)=0$, where
$G:\mathbb{R}^{2}\rightarrow\mathbb{R}$
is precisely defined by $G(\vartheta,\lambda):=\langle{\Psi,\mathcal{T}(\vartheta\varphi_{1},\lambda)}\rangle_{1}
=\langle{\Psi,T(\vartheta\varphi_{1}+h(\vartheta\varphi_{1},\lambda),\lambda)}\rangle_{1}$. Obviously, $G(0,\lambda)=0$.
And it follows from Taylor Formula that $G: \mathbb{R}^{2}\rightarrow\mathbb{R}$ takes the form
\begin{eqnarray}\label{eq19}
G(\vartheta,\lambda)=\vartheta\bigg[\varrho(\lambda-\lambda_{1})+\frac{\kappa}{2}\vartheta+\frac{\nu}{6}\vartheta^{2}+
o\big(\vartheta^{3},\vartheta(\lambda-\lambda_{1})\big)\bigg],
\end{eqnarray}
where
\begin{eqnarray*}
\begin{cases}
     \varrho=\langle{\Psi,T_{\lambda{u}}[\varphi_{1}]}\rangle_{1},  \\
    \kappa=\langle{\Psi,T_{u{u}}[\varphi_{1},\varphi_{1}]}\rangle_{1},     \\
    \nu=\langle{\Psi,T_{uu{u}}[\varphi_{1},\varphi_{1},\varphi_{1}]}\rangle_{1}
    +3\langle{\Psi,T_{u{u}}[\varphi_{1},h_{u_{0}u_{0}}[\varphi_{1},\varphi_{1}]]}\rangle_{1}.
\end{cases}
\end{eqnarray*}

Here, the bilinear from $T_{uu}[\cdot,\cdot]$ and trilinear form $T_{uuu}[\cdot,\cdot,\cdot]$ denote
the second and third order Fr$\acute{e}$chet derivatives of $T$ with respect to $u$,
evaluated at $(u,\lambda)=(0,\lambda_{1})$, respectively. Let $h_{u_{0}}$ and $h_{u_{0}u_{0}}$ be
the first and second order Fr$\acute{e}$chet derivatives of $h$ with respect to $u_{0}$, evaluated at $(u_{0},\lambda)=(0,\lambda_{1})$, respectively.
Obviously, $h_{u_{0}}=0$. We still need to get $h_{u_{0}u_{0}}[\varphi_{1},\varphi_{1}]$.
It follows from (\ref{eq17}) that $\mathcal{L}_{\lambda_{1}}h_{u_{0}u_{0}}[\varphi_{1},\varphi_{1}]+(I-Q_{\mathcal{Y}})T_{uu}[\varphi_{1},\varphi_{1}]=0$ and
hence that
\begin{eqnarray}\label{eq20}
h_{u_{0}u_{0}}[\varphi_{1},\varphi_{1}]
=-\mathcal{L}_{\lambda_{1}}^{-1}(I-Q_{\mathcal{Y}})T_{uu}[\varphi_{1},\varphi_{1}].
\end{eqnarray}

In the sequel, there exist two cases to study the existence of nontrivial zeros of $G(\cdot,\lambda)$. We first consider the case where $\kappa\neq{0}$.
It follows from the implicit function theorem that there exist a constant $\varepsilon>0$ and a continuously differentiable mapping $\vartheta: (\lambda_{1}-\varepsilon,\lambda_{1}+\varepsilon)\rightarrow\mathbb{R}$ such that $G(\vartheta_{\lambda},\lambda)\equiv{0}$ for $\lambda\in(\lambda_{1}-\varepsilon,\lambda_{1}+\varepsilon)$. In fact, we obtain
\begin{eqnarray*}
\vartheta_{\lambda}=\frac{2\varrho(\lambda_{1}-\lambda)}{\kappa}+o\big(|\lambda-\lambda_{1}|\big).
\end{eqnarray*}
Therefore, (\ref{eq4}) has a nontrivial solution $u_{\lambda}=\vartheta_{\lambda}\varphi_{1}+h(\vartheta_{\lambda}\varphi_{1},\lambda)$, which exists for $\lambda\in(\lambda_{1}-\varepsilon,\lambda_{1})\cup(\lambda_{1},\lambda_{1}+\varepsilon)$ and satisfies
$\displaystyle\lim_{\lambda\rightarrow \lambda_{1}}u_{\lambda}=0$.

We next investigate the case where $\kappa={0}$ and $\nu\neq{0}$, which could obtain the multiplicity of nontrivial solutions of (\ref{eq4}).
Then the nontrivial zeros of $G(\cdot,\lambda)$ undergo a saddle-node bifurcation near $\lambda_{1}$. More precisely, if $\varrho\nu<0$
\text{(resp., $\varrho\nu>0$)}, then near the origin, only two nontrivial zeros $\vartheta=\vartheta^{\pm}_{\lambda}$ of $G(\cdot,\lambda)$ exist for $\lambda>\lambda_{1}$ \text{(resp., $\lambda<\lambda_{1}$)}, and no nontrivial zeros exist for $\lambda\leq\lambda_{1}$ \text{(resp., $\lambda\geq\lambda_{1}$)}.
Meanwhile, we have $T_{u{u}}[\varphi_{1},\varphi_{1}]\in\mathcal{R}(\mathcal{L}_{\lambda_{1}})$ due to $\kappa={0}$. And it is known that the mapping $\mathcal{L}_{\lambda_{1}}: \mathcal{X}_{1}\rightarrow\mathcal{R}(\mathcal{L}_{\lambda_{1}})$ is invertible. Hence, there exists precisely a $\varsigma\in\mathcal{X}_{1}$
such that $\mathcal{L}_{\lambda_{1}}\varsigma=-T_{u{u}}[\varphi_{1},\varphi_{1}]$. Then it follows from (\ref{eq20}) that $h_{u_{0}u_{0}}[\varphi_{1},\varphi_{1}]=\varsigma$. Further, the three quantities $\varrho$, $\kappa$ and $\nu$ can be listed as follows:
\begin{eqnarray*}
\begin{cases}
     \varrho=\int_{\Omega}\varphi_{1}^{2}(x)f(x,0)dx+g_{u}(0)\int_{\partial{\Omega}}\varphi_{1}^{2}(x)r(x)dS,  \\
    \kappa=2d\int_{\Omega}\varphi_{1}(x)\nabla\cdot({\varphi_{1}(x)}\nabla{\varphi_{1}(x)})dx
    +2\lambda_{1}\int_{\Omega}\varphi_{1}^{3}(x)f_{u}(x,0)dx+\lambda_{1}g_{uu}(0)\int_{\partial{\Omega}}\varphi_{1}^{3}(x)r(x)dS,     \\
    \nu=3\lambda_{1}\int_{\Omega}\varphi_{1}^{4}(x)f_{uu}(x,0)dx+\lambda_{1}g_{uuu}(0)\int_{\partial{\Omega}}\varphi_{1}^{4}(x)r(x)dS
       +3d\int_{\Omega}\varphi_{1}(x)\nabla\cdot({\varphi_{1}(x)}\nabla{\varsigma(x)})dx+\\
 ~~  3d\int_{\Omega}\varphi_{1}(x)\nabla\cdot({\varsigma(x)}\nabla{\varphi_{1}(x)})dx
    +6\lambda_{1}\int_{\Omega}\varphi_{1}^{2}(x){\varsigma(x)}f_{u}(x,0)dx+3\lambda_{1}g_{uu}(0)\int_{\partial{\Omega}}\varphi_{1}^{2}(x){\varsigma(x)}r(x)dS.
\end{cases}
\end{eqnarray*}

In short, we conclude the following results.

\begin{theorem}\label{th2.1}
Suppose that $(A0)$, (\ref{eq12}) and (\ref{eq13}) are satisfied.
\begin{itemize}
\item[(i)] If $\kappa\neq{0}$, then there exist a constant $\varepsilon$ and a continuously differentiable mapping $\lambda\mapsto{\vartheta_{\lambda}}$ from $(\lambda_{1}-\varepsilon,\lambda_{1}+\varepsilon)$ to $\mathbb{R}$ such that
(\ref{eq4}) has a nontrivial solution $u_{\lambda}=\vartheta_{\lambda}\varphi_{1}+h(\vartheta_{\lambda}\varphi_{1},\lambda)$, which exists for $\lambda\in(\lambda_{1}-\varepsilon,\lambda_{1})\cup(\lambda_{1},\lambda_{1}+\varepsilon)$ and satisfies
$\displaystyle\lim_{\lambda\rightarrow \lambda_{1}}u_{\lambda}=0$.

\item[(ii)] If $\kappa={0}$ and $\varrho\nu<0$
\text{(resp., $\varrho\nu>0$)}, then there exist a constant $\lambda_{*}>\lambda_{1}$ \text{(resp., $\lambda_{*}<\lambda_{1}$)} and two continuously differentiable mappings $\lambda\mapsto{\vartheta_{\lambda}^{\pm}}$ from $[\lambda_{1},\lambda_{*}]$ to $\mathbb{R}$
\text{(resp., from $[\lambda_{*},\lambda_{1}]$ to $\mathbb{R}$)} such that
(\ref{eq4}) has two nontrivial solutions $u_{\lambda}^{\pm}=\vartheta_{\lambda}^{\pm}\varphi_{1}+h(\vartheta_{\lambda}^{\pm}\varphi_{1},\lambda)$, which exists for $\lambda\in(\lambda_{1},\lambda_{*}]$ \text{(resp., $[\lambda_{*},\lambda_{1})$)} and satisfies
$\displaystyle\lim_{\lambda\rightarrow \lambda_{1}}u_{\lambda}^{\pm}=0$.

\end{itemize}
\end{theorem}

\begin{remark}\label{remark1}
From the above discussions and Theorem \ref{th2.1} we note that as for $(i)$ of Theorem \ref{th2.1}, (\ref{eq4}) has a positive nontrivial solution when $\varrho\kappa>0$
\text{$($resp., $\varrho\kappa<0$$)$} for $\lambda\in(\lambda_{1}-\varepsilon,\lambda_{1})$
\text{$($resp., $\lambda\in(\lambda_{1},\lambda_{1}+\varepsilon)$$)$};
$(ii)$ of Theorem \ref{th2.1} may have a negative solution bifurcating from the line of $\Gamma_{0}$ if $\kappa={0}$ and $\varrho\nu<0$
\text{$($resp., $\varrho\nu>0$$)$}, while there is no biological significance for the negative solution.
Throughout the following, we only consider the positive solution, and denote the positive solutions in Theorem \ref{th2.1} by $\lambda\in\Upsilon$.
\end{remark}

\subsection{Local bifurcation from $\Gamma_{u_{1}}$}

~\quad In the subsection, we investigate the bifurcation on $\Gamma_{u_{1}}=\big\{(u_{1},0): u_{1}\in\mathbb{R_{+}}\big\}$.
Note that $T(u_{1},0)\equiv0$ for any constant $u_{1}\in\mathbb{R_{+}}$. From (\ref{eq9}), we obtain
\begin{equation}\label{eq21}
T_{u}(u_{1},0)[v]:=D_{u}T(u_{1},0)[v]=
\left(\begin{array}{c}
   (1+du_{1})\triangle{v}\\
   \partial_{\overrightarrow{n}}v
\end{array}\right),
\end{equation}
which implies that $\mathcal{N}(T_{u}(u_{1},0))=\text{span}\{1\}$.

It is easy to see that codim$\mathcal{R}(T_{u}(u_{1},0))=1$.
In fact, let $(\widetilde{y}_{1},\widetilde{y}_{2})^{T}\in\mathcal{R}(\mathcal{L}_{\lambda_{1}})$ and $\widetilde{v}\in\mathcal{X}$ satisfy
\begin{equation}\label{eq22}
 \begin{cases}
    (1+du_{1})\triangle{\widetilde{v}}=\widetilde{y}_{1},                             &\text{$x \in \Omega$},\\
     \partial_{\overrightarrow{n}}\widetilde{v}=\widetilde{y}_{2},    &\text{$x \in \partial{\Omega}$},
  \end{cases}
\end{equation}
then we have
\begin{eqnarray}\label{eq23}
\begin{split}
\int_{\Omega}\widetilde{y}_{1}dx
=&(1+du_{1})\int_{\Omega}\triangle{\widetilde{v}}dx=(1+du_{1})\int_{\partial{\Omega}}\frac{\partial{\widetilde{v}}}{\partial{\overrightarrow{n}}}dS
=(1+du_{1})\int_{\partial{\Omega}}\widetilde{y}_{2}dS.
\end{split}
\end{eqnarray}
Hence this shows that $(\widetilde{y}_{1},\widetilde{y}_{2})^{T}\in\mathcal{R}(T_{u}(u_{1},0))$ if and only if
\begin{eqnarray}\label{eq24}
\int_{\Omega}\widetilde{y}_{1}(x)dx-(1+du_{1})\int_{\partial{\Omega}}\widetilde{y}_{2}(x)dS=0.
\end{eqnarray}
Therefore, codim$\mathcal{R}(T_{u}(u_{1},0))=1$ and $T_{u}(u_{1},0)$ is a Fredholm operator with index zero.
Thus, we have the decompositions: $\mathcal{X}=\mathcal{N}(T_{u}(u_{1},0))\oplus{\mathcal{X}_{2}}$ and  $\mathcal{Y}=\mathcal{R}(T_{u}(u_{1},0))\oplus{\mathcal{Y}_{2}}$. In the sequel, we denote $\bar{l}\in\mathcal{Y}^{*}$ such that
$\mathcal{R}(T_{u}(u_{1},0))=\big\{\widetilde{y}\in{\mathcal{Y}}, \left \langle {\bar{l},\widetilde{y}} \right \rangle_{2}:=\int_{\Omega}\widetilde{y}_{1}(x)dx-(1+du_{1})\int_{\partial{\Omega}}\widetilde{y}_{2}(x)dS=0\big\}$.
And it follows from Lemma \ref{le1} and (\cite{2007shi-JFA} Theorem 2.1) that we have the following results.

\begin{theorem}\label{th2.2}
Suppose that $(A0)$ is satisfied.
\begin{itemize}
\item[(i)] If $(A1)$ for any constant $u_{1}\in\mathbb{R_{+}}$ is not satisfied, then no steady-state bifurcation occurs in the vicinity of
$(u,\lambda)=(u_{1},0)$, which implies that the steady-state solution set of (\ref{eq4}) near $(u_{1},0)$ consists precisely of the trivial curve $\Gamma_{u_{1}}$.

\item[(ii)] If $(A1)$ and $(A2)$ for any constant $u_{1}\in\mathbb{R_{+}}$ are satisfied,
\begin{eqnarray*}
(A2) ~~d\int_{\Omega}\triangle{\psi_{*}}dx+\int_{\Omega}f(x,u_{1})dx+u_{1}\int_{\Omega}f_{u}(x,u_{1})dx
+(1+du_{1})g_{u}(u_{1})\int_{\partial{\Omega}}r(x)dS\neq0,~~~~~~~~~~~~~~~~~~~~~~~~~
\end{eqnarray*}
where $\psi_{*}$ is the unique solution of
\begin{eqnarray}\label{eq25}
\begin{cases}
     (1+du_{1})\triangle{\widetilde{v}}+u_{1}f(x,u_{1})=0,                             &\text{$x \in \Omega$},\\
     \partial_{\overrightarrow{n}}\widetilde{v}=r(x)g(u_{1}),    &\text{$x \in \partial{\Omega}$},   \\
      \int_{\Omega}\widetilde{v}(x)dx=0,
\end{cases}
\end{eqnarray}
then the set of solutions of (\ref{eq4}) near $(u_{1},0)$ consists explicitly of the curves $\Gamma_{u_{1}}$ and $\Gamma_{\Xi}$, where
\begin{eqnarray*}
\Gamma_{\Xi}=\big\{(u_{01}(s),\lambda_{01}(s)):s\in(-\varepsilon_{1}, \varepsilon_{1}),  \varepsilon_{1}\in\mathbb{R_{+}}\big\},
\end{eqnarray*}
and $u_{01}(s)=u_{1}+\eta_{1}s+s\zeta_{1}(s)$, $\lambda_{01}(s)=s+s\xi_{1}(s)$ are $C^{1}$ functions such that
$\zeta_{1}(0)=\zeta_{1}^{'}(0)=\xi_{1}(0)=\xi_{1}^{'}(0)=0$, and $\eta_{1}=-\frac{\zeta_{*}}{\xi_{*}}$ with
\begin{small}
\begin{eqnarray*}
\begin{split}
\zeta_{*}=&d\int_{\Omega}\nabla\cdot({\psi_{*}}\nabla{\psi_{*}})dx+\int_{\Omega}f(x,u_{1}){\psi_{*}}dx+u_{1}\int_{\Omega}f_{u}(x,u_{1}){\psi_{*}}dx
+(1+du_{1})g_{u}(u_{1})\int_{\partial{\Omega}}r(x){\psi_{*}}dS,  \\
\xi_{*}=&d\int_{\Omega}\triangle{\psi_{*}}dx+\int_{\Omega}f(x,u_{1})dx+u_{1}\int_{\Omega}f_{u}(x,u_{1})dx
+(1+du_{1})g_{u}(u_{1})\int_{\partial{\Omega}}r(x)dS.
\end{split}
\end{eqnarray*}
\end{small}

\end{itemize}
\end{theorem}
\begin{proof}
Take the partial derivative of $T(u,\lambda)$ with respect to $\lambda$, evaluated at $(u_{1},0)$, which is given by
\begin{equation}\label{eq26}
T_{\lambda}(u_{1},0)=
\left(\begin{array}{c}
   u_{1}f(x,u_{1})\\
   -r(x)g(u_{1})
\end{array}\right).
\end{equation}

$(i)$  If $(A1)$ for any constant $u_{1}\in\mathbb{R_{+}}$ is not satisfied,
then it follows from (\ref{eq26}) that $T_{\lambda}(u_{1},0)\notin\mathcal{R}(T_{u}(u_{1},0))$.
It follows from Lemma \ref{le1} and the implicit function theorem that $(u_{1},0)$ with $u_{1}\in\mathbb{R_{+}}$ is not a bifurcation point of (\ref{eq4})
regarding the trivial branch $\Gamma_{u_{1}}$.

$(ii)$  If $(A1)$ for $u_{1}\in\mathbb{R_{+}}$ holds, then we get $T_{\lambda}(u_{1},0)\in\mathcal{R}(T_{u}(u_{1},0))$ from (\ref{eq26}).
Applying Theorem 2.1 in \cite{2007shi-JFA} to $T(u,\lambda)=0$, we define ${\mathcal{X}_{2}}:=\big\{\widetilde{v}\in{\mathcal{X}}, \int_{\Omega}\widetilde{v}(x)dx=0\big\}$, and
obtain the unique solution $\psi_{*}$ of
$T_{\lambda}(u_{1},0)+T_{u}(u_{1},0)[\widetilde{v}]=0$ for $\widetilde{v}\in{\mathcal{X}_{2}}$, then $\psi_{*}$ satisfies  (\ref{eq25}).

Then there exists the Hessian matrix $G_{0}$ evaluated at $(u,\lambda)=(u_{1},0)$,
\begin{equation}\label{eq27}
G_{0}= {\left(\begin{array}{cc}
   \left \langle {\bar{l},2T_{\lambda{u}}(u_{1},0)[\psi_{*}]+T_{uu}(u_{1},0)[\psi_{*}]^{2}} \right \rangle_{2} &
   \left \langle {\bar{l},T_{\lambda{u}}(u_{1},0)[1]+T_{uu}(u_{1},0)[1,\psi_{*}]} \right \rangle_{2} \\
   \left \langle {\bar{l},T_{\lambda{u}}(u_{1},0)[1]+T_{uu}(u_{1},0)[1,\psi_{*}]} \right \rangle_{2}       & 0
    \end{array}\right)},
\end{equation}
where
\begin{eqnarray*}
\begin{split}
& \left \langle {\bar{l},2T_{\lambda{u}}(u_{1},0)[\psi_{*}]} \right \rangle_{2}
=2\bigg[\int_{\Omega}f(x,u_{1}){\psi_{*}}dx+u_{1}\int_{\Omega}f_{u}(x,u_{1}){\psi_{*}}dx
+(1+du_{1})g_{u}(u_{1})\int_{\partial{\Omega}}r(x){\psi_{*}}dS\bigg],  \\
& \left \langle {\bar{l},T_{uu}(u_{1},0)[\psi_{*}]^{2}} \right \rangle_{2}
=2d\int_{\Omega}\nabla\cdot({\psi_{*}}\nabla{\psi_{*}})dx=2dg(u_{1})\int_{\partial{\Omega}}r(x){\psi_{*}}dS,  \\
& \left \langle {\bar{l},T_{\lambda{u}}(u_{1},0)[1]} \right \rangle_{2}
=\int_{\Omega}f(x,u_{1})dx+u_{1}\int_{\Omega}f_{u}(x,u_{1})dx
+(1+du_{1})g_{u}(u_{1})\int_{\partial{\Omega}}r(x)dS,  \\
& \left \langle {\bar{l},T_{uu}(u_{1},0)[1,\psi_{*}]} \right \rangle_{2}
=d\int_{\Omega}\triangle{\psi_{*}}dx=dg(u_{1})\int_{\partial{\Omega}}r(x)dS,
\end{split}
\end{eqnarray*}
by using that $\left \langle {\bar{l},T_{\lambda{\lambda}}(u_{1},0)} \right \rangle_{2}=\left \langle {\bar{l},T_{uu}(u_{1},0)[1]^{2}} \right \rangle_{2}=0$.
Then we get the following result from $(A2)$
\begin{eqnarray*}
\det(G_{0})=-\bigg[d\int_{\Omega}\triangle{\psi_{*}}dx+\int_{\Omega}f(x,u_{1})dx+u_{1}\int_{\Omega}f_{u}(x,u_{1})dx
+(1+du_{1})g_{u}(u_{1})\int_{\partial{\Omega}}r(x)dS\bigg]^{2}<0.
\end{eqnarray*}

Hence it follows from Theorem 2.1 in \cite{2007shi-JFA} that the solution set of $T(u,\lambda)=0$ near $(u,\lambda)=(u_{1},0)$
is the union of two intersecting $C^{1}$ curves, which  are in form of
$(u_{0i}(s),\lambda_{0i}(s))=(u_{1}+\eta_{i}s+s\zeta_{i}(s),\widetilde{\eta}_{i}s+s\xi_{i}(s)), s\in(-\varepsilon_{1}, \varepsilon_{1}),  \varepsilon_{1}\in\mathbb{R_{+}}, \zeta_{i}(s)\in{\mathcal{X}_{2}}, i=1,2$, where $(\widetilde{\eta}_{1},\eta_{1})$ and $(\widetilde{\eta}_{2},\eta_{2})$ are non-zero linear independent solutions
of the following equation
\begin{eqnarray*}
\left \langle {\bar{l},2T_{\lambda{u}}(u_{1},0)[\psi_{*}]+T_{uu}(u_{1},0)[\psi_{*}]^{2}} \right \rangle_{2}
\widetilde{\eta}^{2}
+2\left \langle {\bar{l},T_{\lambda{u}}(u_{1},0)[1]+T_{uu}(u_{1},0)[1,\psi_{*}]} \right \rangle_{2}
\widetilde{\eta}\eta=0,
\end{eqnarray*}
and it is shown that $(\widetilde{\eta}_{1},\eta_{1})=(1,\eta_{1})$ and $(\widetilde{\eta}_{2},\eta_{2})=(0,1)$ with $\eta_{1}=-\frac{\zeta_{*}}{\xi_{*}}$, $\zeta_{1}(0)=\zeta_{1}^{'}(0)=\xi_{1}(0)=\xi_{1}^{'}(0)=0$.
Thus, the solution set of (\ref{eq4}) near $(u_{1},0)$ consists precisely of the curves
$(u_{01}(s),\lambda_{01}(s))=(u_{1}+\eta_{1}s+s\zeta_{1}(s),s+s\xi_{1}(s))$ and $(u_{02}(s),\lambda_{02}(s))=(u_{1}+s+s\zeta_{2}(s),s\xi_{2}(s))$,
while the solution curve $(u_{02}(s),\lambda_{02}(s))$ is identical to the trivial curve $\Gamma_{u_{1}}$.
\end{proof}

\section{Stability of steady state solutions in Theorem \ref{th2.1}}
Theorem \ref{th2.1} states that there exists an open set $\Upsilon\subseteq{\mathbb{R}}^{+}$ with $\lambda_{1}$ on its boundary such that model (\ref{eq3}) with $\lambda\in{\Upsilon}$ has a spatially nonhomogeneous positive steady state solution $u_{\lambda}^{*}$. It follows from the previous discussions that $u_{\lambda}^{*}$ takes the form $u_{\lambda}^{*}=\vartheta_{\lambda}^{*}\varphi_{1}+h(\vartheta_{\lambda}^{*}\varphi_{1},\lambda)$, where $\vartheta=\vartheta_{\lambda}^{*}$ is a solution to $G(\vartheta,\lambda)=0$ given in Section 2.1. Obviously, we have $\vartheta_{\lambda}^{*}(\lambda_{1})=0$ and
$\displaystyle\lim_{\lambda\rightarrow \lambda_{1}}\frac{u_{\lambda}^{*}}{\vartheta_{\lambda}^{*}}=\varphi_{1}$. In the following, we will study the stability of the bifurcated steady state solution $u_{\lambda}^{*}$, when $\Omega$ is a bounded open set in $\mathbb{R}$. To investigate the local dynamical behavior of model (\ref{eq3}) near $u=u_{\lambda}^{*}$, the linearized equation of (\ref{eq3}) at $u_{\lambda}^{*}$ is given by
\begin{equation}\label{eq28}
 \begin{cases}
    \frac{\partial{v}}{\partial{t}}=\triangle{v}+d\nabla\cdot({v}\nabla{u_{\lambda}^{*}})+d\nabla\cdot({u_{\lambda}^{*}}\nabla{v_{\sigma}})
    +{\lambda}u_{\lambda}^{*}f_{u}(x,u_{\lambda}^{*})v+{\lambda}f(x,u_{\lambda}^{*})v,     &\text{$x \in \Omega, t >0 $},\\
     \partial_{\overrightarrow{n}}v={\lambda}r(x)g_{u}(u_{\lambda}^{*})v,                                    &\text{$x \in \partial{\Omega}, t >0 $},
  \end{cases}
\end{equation}
where $v=v(x,t)$ and $v_{\sigma}=v(x,t-\sigma)$.
From \cite{1996-wujianhong-pfde,1983-A.Pazy,2020-DCDSA-anqi-memory-nonlocal,2021-JDDE-wangchuncheng-memory},
we are looking for $\mu\in{\mathbb{C}}$ and $\psi\in\mathcal{X}_{\mathbb{C}}\setminus{\{0\}}$ with any $\lambda\in\Upsilon$ and $\sigma\in{\mathbb{R}}\setminus{\mathbb{R_{-}}}$, such that
\begin{eqnarray}\label{eq29}
\triangle(\lambda,\mu,\sigma)\psi(x)=0,          &\text{$x \in \Omega$},
\end{eqnarray}
where
\begin{eqnarray*}
\triangle(\lambda,\mu,\sigma)\psi=
   \triangle{\psi}+d\nabla\cdot({\psi}\nabla{u_{\lambda}^{*}})+d\nabla\cdot({u_{\lambda}^{*}}\nabla{\psi})e^{-\mu\sigma}
    +{\lambda}u_{\lambda}^{*}f_{u}(x,u_{\lambda}^{*})\psi+{\lambda}f(x,u_{\lambda}^{*})\psi-\mu{\psi},
\end{eqnarray*}
and for all $\psi\in\mathcal{X}_{\mathbb{C}}\setminus{\{0\}}$ satisfying $\partial_{\overrightarrow{n}}\psi(x)={\lambda}r(x)g_{u}(u_{\lambda}^{*})\psi(x)$ on $x \in \partial{\Omega}$.

If any $(\lambda,\mu,\sigma,\psi)\in\Upsilon\times{\mathbb{C}}\times{{\mathbb{R}}\setminus{\mathbb{R_{-}}}}\times{\mathcal{X}_{\mathbb{C}}\setminus{\{0\}}}$ can solve (\ref{eq29}) and satisfy $\partial_{\overrightarrow{n}}\psi={\lambda}r(x)g_{u}(u_{\lambda}^{*})\psi$ on $\partial{\Omega}$, the complex number $\mu$ can be referred to as an eigenvalue of (\ref{eq28}) associated with eigenfunction $\psi$, then without loss of generality, we assume that $\|{\psi}\|_{\mathcal{Y}_{\mathbb{C}}}=1$ for simplicity. Now, we respectively give two prior estimates for the eigenfunction $\psi_{\lambda}$ and eigenvalue $\mu_{\lambda}$ before we discuss the distribution of the eigenvalue $\psi_{\lambda}$ of (\ref{eq28}) associated with the eigenfunction $\mu_{\lambda}$ for $\lambda\in\Upsilon$. Firstly, we have the following lemma on the estimate for the eigenfunction.

\begin{lemma}\label{le3}
Suppose that $(A0)$, $(A3)$, (\ref{eq12}) and (\ref{eq13}) are satisfied, where
\begin{eqnarray*}
(A3) ~~|d|<d_{*}:=\frac{1}{\displaystyle\max_{\lambda\in{\Upsilon}}\max_{x\in\bar{\Omega}}\big\{{u_{\lambda}^{*}(x)}\big\}},~~
~~~~~~~~~~~~~~~~~~~~~~~~~~~~~~~~~~~~~~~~~~~~~~~~~~~~~~~~~~~~~~~~~~~~~~~~~
\end{eqnarray*}
then there exists a constant $\mathbf{C}$, such that for any $(\lambda,\mu_{\lambda},\sigma_{\lambda},\psi_{\lambda})\in{\Upsilon}\times{\mathbb{C}}\times{\mathbb{R}}\setminus{\mathbb{R_{-}}}\times{\mathcal{X}_{\mathbb{C}}\setminus{\{0\}}}$ with $Re\{\mu_{\lambda}\}\geq 0$ solving (\ref{eq29}) and satisfying $\partial_{\overrightarrow{n}}\psi_{\lambda}={\lambda}r(x)g_{u}(u_{\lambda}^{*})\psi_{\lambda}$ on $\partial{\Omega}$, and
\begin{eqnarray*}
  \big\|{\nabla\psi_{\lambda}}\big\|_{\mathcal{Y}_{\mathbb{C}}}^{2}\leq{\mathbf{C}}.
  \end{eqnarray*}
\end{lemma}
\begin{proof}
According to the continuity of $\lambda\mapsto{\vartheta_{\lambda}}^{*}$, it follows that ${\vartheta_{\lambda}}^{*}\in\mathbb{R}$, $h(\vartheta_{\lambda}^{*}\varphi_{1},\lambda)\in{\mathcal{X}}_{1}$ and $u_{\lambda}^{*}\in{\mathcal{X}}$ are bounded for any $\lambda\in\Upsilon$. Note that $u_{\lambda}^{*}$ is a solution of (\ref{eq4}), by the embedding theorem \cite{2003-Sobolev Spaces(Adams)} and regularity theory for elliptic equations \cite{2001-Elliptic-PDE-Second Order}, we get $h(\vartheta_{\lambda}^{*}\varphi_{1},\lambda)\in{C^{2,\check{\alpha}}}, u_{\lambda}^{*}\in{C^{2,\check{\alpha}}}$, where $0<\check{\alpha}<\min\{\tilde{\alpha},\frac{1}{2}\}$.

Since
\begin{eqnarray*}
\begin{split}
\langle{\psi_{\lambda},\nabla\cdot({\psi_{\lambda}}\nabla{u_{\lambda}^{*}})}\rangle
&=\int_{\Omega}\bar{\psi}_{\lambda}\nabla\cdot({\psi_{\lambda}}\nabla{u_{\lambda}^{*}})dx
=\int_{\Omega}\bar{\psi}_{\lambda}\nabla{\psi_{\lambda}}\cdot\nabla{u_{\lambda}^{*}}dx+\int_{\Omega}|\psi_{\lambda}|^{2}\triangle{u_{\lambda}^{*}}dx  \\
&=\int_{\partial{\Omega}}|\psi_{\lambda}|^{2}\frac{\partial{u_{\lambda}^{*}}}{\partial{\overrightarrow{n}}}dS
-\int_{\Omega}\psi_{\lambda}\nabla\cdot({\bar{\psi}_{\lambda}}\nabla{u_{\lambda}^{*}})dx+\int_{\Omega}|\psi_{\lambda}|^{2}\triangle{u_{\lambda}^{*}}dx,
\end{split}
\end{eqnarray*}
we have
\begin{eqnarray}\label{eq30}
\begin{split}
Re\Big\{\langle{\psi_{\lambda},\nabla\cdot({\psi_{\lambda}}\nabla{u_{\lambda}^{*}})}\rangle\Big\}
=\frac{1}{2}\int_{\partial{\Omega}}|\psi_{\lambda}|^{2}\frac{\partial{u_{\lambda}^{*}}}{\partial{\overrightarrow{n}}}dS
+\frac{1}{2}\int_{\Omega}|\psi_{\lambda}|^{2}\triangle{u_{\lambda}^{*}}dx.
\end{split}
\end{eqnarray}

Multiplying two sides of $\triangle(\lambda,\mu_{\lambda},\sigma_{\lambda})\psi_{\lambda}=0$ by $\bar{\psi}_{\lambda}$, and integrating over $\Omega$,
we yield
\begin{eqnarray}\label{eq31}
\begin{split}
\big\|{\nabla{\psi_\lambda}}\big\|_{\mathcal{Y}_{\mathbb{C}}}^{2}=&Re\bigg\{\int_{\partial{\Omega}}\bar{\psi}_{\lambda}\frac{\partial{\psi_{\lambda}}}{\partial{\overrightarrow{n}}}dS\bigg\}
+dRe\Big\{\langle{\psi_{\lambda},\nabla\cdot({\psi_{\lambda}}\nabla{u_{\lambda}^{*}})}\rangle\Big\}
+dRe\Big\{\langle{\psi_{\lambda},\nabla\cdot({u_{\lambda}^{*}}\nabla{\psi_{\lambda}})}\rangle{e^{-\mu_{\lambda}{\sigma_{\lambda}}}}\Big\}\\
&+{\lambda}Re\Big\{\langle{\psi_{\lambda},{\psi_{\lambda}}[u_{\lambda}^{*}f_{u}(x,u_{\lambda}^{*})+f(x,u_{\lambda}^{*})]}\rangle\Big\}
-Re\Big\{\mu_{\lambda}\Big\},
\end{split}
\end{eqnarray}
then it follows from (\ref{eq29}) and (\ref{eq30}) with $Re\{\mu_{\lambda}\}\geq{0}$ that
\begin{eqnarray}\label{eq32}
\begin{split}
\big\|{\nabla{\psi_\lambda}}\big\|_{\mathcal{Y}_{\mathbb{C}}}^{2}=&\lambda\int_{\partial{\Omega}}|\psi_{\lambda}|^{2}r(x)g_{u}(u_{\lambda}^{*})dS+
\frac{d}{2}\int_{\partial{\Omega}}|\psi_{\lambda}|^{2}\frac{\partial{u_{\lambda}^{*}}}{\partial{\overrightarrow{n}}}dS
+\frac{d}{2}\int_{\Omega}|\psi_{\lambda}|^{2}\triangle{u_{\lambda}^{*}}dx\\
&-d\int_{\Omega}{u_{\lambda}^{*}}|{\nabla{\psi_\lambda}}|^{2}dx{Re}\Big\{{e^{-\mu_{\lambda}{\sigma_{\lambda}}}}\Big\}
+d\lambda\int_{\partial{\Omega}}|\psi_{\lambda}|^{2}r(x)g_{u}(u_{\lambda}^{*})u_{\lambda}^{*}dS{Re}\Big\{{e^{-\mu_{\lambda}{\sigma_{\lambda}}}}\Big\}\\
&+{\lambda}Re\Big\{\langle{\psi_{\lambda},{\psi_{\lambda}}[u_{\lambda}^{*}f_{u}(x,u_{\lambda}^{*})+f(x,u_{\lambda}^{*})]}\rangle\Big\}
-Re\Big\{\mu_{\lambda}\Big\},\\
{\leq}&\big\|\lambda{r(x)g_{u}(u_{\lambda}^{*})}\big\|_{\infty}+\frac{|d|}{2}\big\|\lambda{r(x)g(u_{\lambda}^{*})}\big\|_{\infty}
+\frac{|d|}{2}\big\|\triangle{u_{\lambda}^{*}}\big\|_{\infty}
+|d|{\displaystyle\max_{x\in\bar{\Omega}}\big\{{u_{\lambda}^{*}(x)}\big\}}\big\|{\nabla{\psi_\lambda}}\big\|_{\mathcal{Y}_{\mathbb{C}}}^{2}\\
&+|d|\big\|\lambda{r(x)g_{u}(u_{\lambda}^{*})u_{\lambda}^{*}}\big\|_{\infty}
+\big\|\lambda{u_{\lambda}^{*}f_{u}(x,u_{\lambda}^{*})}+\lambda{f(x,u_{\lambda}^{*})}\big\|_{\infty},\\
:=&{\mathbf{C}_{0}}+|d|{\displaystyle\max_{x\in\bar{\Omega}}\big\{{u_{\lambda}^{*}(x)}\big\}}\big\|{\nabla{\psi_\lambda}}\big\|_{\mathcal{Y}_{\mathbb{C}}}^{2},
\end{split}
\end{eqnarray}
thus, it follows from $(A3)$ that we get
\begin{eqnarray}\label{eq33}
\big\|{\nabla{\psi_\lambda}}\big\|_{\mathcal{Y}_{\mathbb{C}}}^{2}
{\leq}\frac{\mathbf{C}_{0}}{1-|d|{\displaystyle\max_{x\in\bar{\Omega}}\big\{{u_{\lambda}^{*}(x)}\big\}}}:={\mathbf{C}}.
\end{eqnarray}

The proof is completed.
\end{proof}

Sequently a priori estimate for the eigenvalue can be obtained in the following lemma.
\begin{lemma}\label{le4}
Suppose that $(A0)$, $(A3)$, (\ref{eq12}) and (\ref{eq13}) are satisfied. If $(\lambda,\mu_{\lambda},\sigma_{\lambda},\psi_{\lambda})\in{\Upsilon}\times{\mathbb{C}}\times{\mathbb{R}}\setminus{\mathbb{R_{-}}}\times{\mathcal{X}_{\mathbb{C}}\setminus{\{0\}}}$ with $Re\{\mu_{\lambda}\}\geq 0$ can solve (\ref{eq29}) and satisfy $\partial_{\overrightarrow{n}}\psi_{\lambda}={\lambda}r(x)g_{u}(u_{\lambda}^{*})\psi_{\lambda}$ on $\partial{\Omega}$,
then $\Big|\frac{\mu_{\lambda}}{\vartheta_{\lambda}^{*}}\Big|$ is bounded for $\lambda\in{\Upsilon}$.
\end{lemma}
\begin{proof}
For each fixed $\lambda\in{\Upsilon}$, define the linear self-conjugate operator $Q_{\lambda}:{\mathcal{X}_{\mathbb{C}}}\rightarrow{\mathcal{Y}_{\mathbb{C}}}$ by
\begin{eqnarray*}
Q_{\lambda}\psi=\triangle{\psi}+d\nabla\cdot({u_{\lambda}^{*}}\nabla{\psi})+{\lambda}f(x,u_{\lambda}^{*})\psi,     \qquad  \forall{\psi}\in{\mathcal{X}_{\mathbb{C}}},
\end{eqnarray*}
and note that $u_{\lambda}^{*}>0$ is a positive solution of (\ref{eq4}), thus $Q_{\lambda}u_{\lambda}^{*}=0$, which implies that $0$ is the principal eigenvalue of $Q_{\lambda}$. Therefore, $\langle{\psi,Q_{\lambda}\psi}\rangle\leq{0}$ for $\forall{\psi}\in{\mathcal{X}_{\mathbb{C}}}$. Then taking the inner product of $\psi_{\lambda}$ with both sides of $\triangle(\lambda,\mu_{\lambda},\sigma_{\lambda})\psi_{\lambda}=0$, we get
\begin{eqnarray}\label{eq34}
\begin{split}
0\geq\langle{\psi_{\lambda},Q_{\lambda}\psi_{\lambda}}\rangle
=&\langle{\psi_{\lambda},\triangle{\psi_{\lambda}}+d\nabla\cdot({u_{\lambda}^{*}}\nabla{\psi_{\lambda}})+{\lambda}f(x,u_{\lambda}^{*})\psi_{\lambda}}\rangle\\
=&\mu_{\lambda}-d\langle{\psi_{\lambda},\nabla\cdot(\psi_{\lambda}\nabla{u_{\lambda}^{*}})}\rangle
-d\big[e^{-\mu_{\lambda}\sigma_{\lambda}}-1\big]\langle{\psi_{\lambda},\nabla\cdot({u_{\lambda}^{*}}\nabla{\psi_{\lambda}})}\rangle\\
&-\langle{\psi_{\lambda},{\lambda}u_{\lambda}^{*}f_{u}(x,u_{\lambda}^{*})\psi_{\lambda}}\rangle.
\end{split}
\end{eqnarray}

By the embedding theorem \cite{2003-Sobolev Spaces(Adams)} and regularity theory for elliptic equations \cite{2001-Elliptic-PDE-Second Order}, we can give the following estimate
from (\ref{eq30}), $Re\{\mu_{\lambda}\}\geq 0$, and Lemma \ref{le3}
\begin{eqnarray*}
\begin{split}
\Bigg|Re\bigg\{\frac{\mu_{\lambda}}{\vartheta_{\lambda}^{*}}\bigg\}\Bigg|\leq
&\Bigg|Re\bigg\{\frac{d}{\vartheta_{\lambda}^{*}}\langle{\psi_{\lambda},\nabla\cdot(\psi_{\lambda}\nabla{u_{\lambda}^{*}})}\rangle\bigg\}\Bigg|+
\Bigg|Re\bigg\{\frac{d}{\vartheta_{\lambda}^{*}}
\big[e^{-\mu_{\lambda}\sigma_{\lambda}}-1\big]\langle{\psi_{\lambda},\nabla\cdot({u_{\lambda}^{*}}\nabla{\psi_{\lambda}})}\rangle\bigg\}\Bigg|\\
&+\Bigg|Re\bigg\{\frac{\lambda}{\vartheta_{\lambda}^{*}}\int_{\Omega}|\psi_{\lambda}|^{2}u_{\lambda}^{*}f_{u}(x,u_{\lambda}^{*})dx\bigg\}\Bigg|\\
\leq
&\bigg|\frac{\lambda{d}}{2\vartheta_{\lambda}^{*}}\bigg|\big\|r(x)g_{u}(0)u_{\lambda}^{*}\big\|_{\infty}
+\bigg|\frac{d}{2\vartheta_{\lambda}^{*}}\bigg|\big\|\triangle{u_{\lambda}^{*}}\big\|_{\infty}+
\bigg|\frac{2d}{\vartheta_{\lambda}^{*}}\bigg|
\Big[{\mathbf{C}}\big\|{u_{\lambda}^{*}}\big\|_{\infty}+\big\|{{\lambda}r(x)g_{u}(u_{\lambda}^{*})u_{\lambda}^{*}}\big\|_{\infty}\Big]\\
&+\bigg|\frac{\lambda}{\vartheta_{\lambda}^{*}}\bigg|\Big\|{u_{\lambda}^{*}f_{u}(x,u_{\lambda}^{*})}\Big\|_{\infty},
\end{split}
\end{eqnarray*}

Similarly, we have
\begin{eqnarray*}
\begin{split}
\Bigg|Im\bigg\{\frac{\mu_{\lambda}}{\vartheta_{\lambda}^{*}}\bigg\}\Bigg|\leq
&\Bigg|Im\bigg\{\frac{d}{\vartheta_{\lambda}^{*}}\langle{\psi_{\lambda},\nabla\psi_{\lambda}\cdot\nabla{u_{\lambda}^{*}}}\rangle\bigg\}\Bigg|+
\Bigg|Im\bigg\{\frac{de^{-\mu_{\lambda}\sigma_{\lambda}}}{\vartheta_{\lambda}^{*}}
\langle{\psi_{\lambda},\nabla\cdot({u_{\lambda}^{*}}\nabla{\psi_{\lambda}})}\rangle\bigg\}\Bigg|\\
\leq
&\bigg|\frac{d\sqrt{\mathbf{C}}}{\vartheta_{\lambda}^{*}}\bigg|\big\|{\nabla{u_{\lambda}^{*}}}\big\|_{\infty}+
\bigg|\frac{d}{\vartheta_{\lambda}^{*}}\bigg|
\Big[{\mathbf{C}}\big\|{u_{\lambda}^{*}}\big\|_{\infty}+\big\|{{\lambda}r(x)g_{u}(u_{\lambda}^{*})u_{\lambda}^{*}}\big\|_{\infty}\Big].
\end{split}
\end{eqnarray*}

Hence, it follows from Theorem \ref{th2.1} that $\lambda\mapsto(\vartheta_{\lambda}^{*},u_{\lambda}^{*})$ is
continuous, it is easy to see that $\Big|\frac{\mu_{\lambda}}{\vartheta_{\lambda}^{*}}\Big|$ is bounded for $\lambda\in\Upsilon$.
\end{proof}

Now we first consider the existence of the zero eigenvalue of (\ref{eq28}). For convenience, define the  notation
\begin{equation}\label{eq35}
m(\lambda,\mu,\sigma)\psi:=
\left(\begin{array}{c}
   \triangle(\lambda,\mu,\sigma)\psi\\
   \partial_{\overrightarrow{n}}\psi-{\lambda}r(x)g_{u}(u_{\lambda}^{*})\psi
\end{array}\right).
\end{equation}

\begin{lemma}\label{le5}
Suppose that $(A0)$, $(A3)$, (\ref{eq12}) and (\ref{eq13}) are satisfied.
\begin{itemize}
\item[(i)] For each $(\sigma,\lambda)\in{\mathbb{R}}\setminus{\mathbb{R_{-}}}\times\Upsilon$, if $0$ is a zero eigenvalue of (\ref{eq28}), then $\kappa=0$.

\item[(ii)] For each $(\sigma,\lambda)\in{\mathbb{R}}\setminus{\mathbb{R_{-}}}\times\Upsilon$, if $\kappa=0$, $T_{u{u}}[\varphi_{1},\varphi_{1}]\neq{0}$,
then $0$ is a zero eigenvalue of (\ref{eq28}).
\end{itemize}
\end{lemma}
\begin{proof}
$(i)$ If $0$ is the zero eigenvalue of (\ref{eq28}), then there exists some $(\lambda,\sigma_{\lambda},\psi_{\lambda})\in{\Upsilon}\times{\mathbb{R}}\setminus{\mathbb{R_{-}}}\times{\mathcal{X}_{\mathbb{C}}\setminus{\{0\}}}$
such that $m(\lambda,0,\sigma_{\lambda})\psi_{\lambda}=0$. Note that the decomposition of space $\mathcal{X}_{\mathbb{C}}=(\mathcal{N}(\mathcal{L}_{\lambda_{1}}))_{\mathbb{C}}\oplus({\mathcal{X}_{1}})_{\mathbb{C}}$,
and $m(\lambda_{1},0,\sigma_{\lambda_{1}})=\mathcal{L}_{\lambda_{1}}$,
and $\mathcal{N}(\mathcal{L}_{\lambda_{1}})=\text{span}\{\varphi_{1}\}$.
After ignoring a scalar factor, $\psi_{\lambda}$ is represented as
\begin{eqnarray}\label{eq36}
\begin{cases}
\psi_{\lambda}=\alpha_{\lambda}\varphi_{1}+\vartheta_{\lambda}^{*}\beta_{\lambda}, ~~\beta_{\lambda}\in({\mathcal{X}_{1}})_{\mathbb{C}},
~~\alpha_{\lambda}\geq 0,\\
\|\psi_{\lambda}\|^{2}_{\mathcal{Y}_{\mathbb{C}}}=
\alpha_{\lambda}^{2}\|{\varphi_{1}}\|^{2}_{\mathcal{Y}_{\mathbb{C}}}+(\vartheta_{\lambda}^{*})^{2}\|{\beta_{\lambda}}\|^{2}_{\mathcal{Y}_{\mathbb{C}}}
=\|{\varphi_{1}}\|^{2}_{\mathcal{Y}_{\mathbb{C}}}=1.
\end{cases}
\end{eqnarray}
Thus, $m(\lambda,0,\sigma_{\lambda})\psi_{\lambda}=
\alpha_{\lambda}m(\lambda,0,\sigma_{\lambda})\varphi_{1}+\vartheta_{\lambda}^{*}m(\lambda,0,\sigma_{\lambda})\beta_{\lambda}=0$. By Taylor expansion and some  calculations, we have
\begin{eqnarray}\label{eq37}
\displaystyle\lim_{\lambda\rightarrow \lambda_{1}}\frac{1}{\vartheta_{\lambda}^{*}}m(\lambda,0,\sigma_{\lambda})\varphi_{1}=
T_{u{u}}[\varphi_{1},\varphi_{1}]+T_{\lambda{u}}[\varphi_{1}]\displaystyle\lim_{\lambda\rightarrow \lambda_{1}}\frac{\lambda-\lambda_{1}}{\vartheta_{\lambda}^{*}}:=\hat{\varphi}_{1},
\end{eqnarray}
and $\langle{\Psi,\hat{\varphi}_{1}}\rangle_{1}=\frac{1}{2}\kappa$. Therefore, we get $\alpha_{\lambda_{1}}\hat{\varphi}_{1}+\mathcal{L}_{\lambda_{1}}\beta_{\lambda_{1}}=0$.
If $\kappa\neq{0}$, then $\hat{\varphi}_{1}\notin\mathcal{R}(\mathcal{L}_{\lambda_{1}})$, which means that $\alpha_{\lambda_{1}}=0$.
Further, it follows from the fact that $\beta_{\lambda_{1}}\in({\mathcal{X}_{1}})_{\mathbb{C}}$ and $\mathcal{L}_{\lambda_{1}}$ is invertible
if it is restricted in $({\mathcal{X}_{1}})_{\mathbb{C}}$ that we have $\beta_{\lambda_{1}}\equiv{0}$. This implies that
$\displaystyle\lim_{\lambda\rightarrow \lambda_{1}}m(\lambda,0,\sigma_{\lambda})\psi_{\lambda}=0$ has only the zero solution, which contradicts
$\psi_{\lambda_{1}}\in{\mathcal{X}_{\mathbb{C}}\setminus{\{0\}}}$. Hence, $\kappa={0}$.

$(ii)$  In order to check that $0$ is the zero eigenvalue of (\ref{eq28}), it suffices to look for a solution of $(\lambda,\sigma,\psi)\in\Upsilon\times{\mathbb{R}}\setminus{\mathbb{R_{-}}}\times{\mathcal{X}_{\mathbb{C}}\setminus{\{0\}}}$ to the equation $m(\lambda,0,\sigma)\psi=0$.
Substituting $\psi=\varphi_{1}+\vartheta_{\lambda}^{*}\beta$ with $\beta\in({\mathcal{X}_{1}})_{\mathbb{C}}$ into $m(\lambda,0,\sigma)\psi=0$, we get
\begin{eqnarray}\label{eq38}
H(\beta,\lambda):=\frac{1}{\vartheta_{\lambda}^{*}}m(\lambda,0,\sigma)\varphi_{1}+m(\lambda,0,\sigma)\beta=0,
\end{eqnarray}
from which we obtain that $H(\beta,\lambda_{1})=T_{u{u}}[\varphi_{1},\varphi_{1}]+\mathcal{L}_{\lambda_{1}}\beta$.
Next, from the assumption that $\kappa=0$ and $T_{u{u}}[\varphi_{1},\varphi_{1}]\neq{0}$,
which could imply that $0\neq{T_{u{u}}[\varphi_{1},\varphi_{1}]}\in\mathcal{R}(\mathcal{L}_{\lambda_{1}})$, there exists $\beta^{*}\in({\mathcal{X}_{1}})_{\mathbb{C}}\setminus{\{0\}}$ such that $T_{u{u}}[\varphi_{1},\varphi_{1}]=-\mathcal{L}_{\lambda_{1}}\beta^{*}$,
$i.e.$, $H(\beta^{*},\lambda_{1})=0$. It is easy to see that ${T_{u{u}}[\varphi_{1},\varphi_{1}]}\neq{0}$ and $H_{\beta}(\beta^{*},\lambda_{1})=\mathcal{L}_{\lambda_{1}}$ is invertible
as it is restricted in $({\mathcal{X}_{1}})_{\mathbb{C}}$. Then by the implicit function theorem, we get a continuously differentiable mapping
$\lambda\mapsto{\beta_{\lambda}}$ from $\Upsilon$ to $({\mathcal{X}_{1}})_{\mathbb{C}}$ such that $\beta_{\lambda_{1}}=\beta^{*}$, and $H(\beta_{\lambda},\lambda)\equiv{0}$. This means that $m(\lambda,0,\sigma)\psi=0$ has a nontrivial solution $\psi=\varphi_{1}+\vartheta_{\lambda}^{*}\beta_{\lambda}\in{\mathcal{X}_{\mathbb{C}}\setminus{\{0\}}}$
for each $(\sigma,\lambda)\in{\mathbb{R}}\setminus{\mathbb{R_{-}}}\times\Upsilon$. Hence, $0$ is the zero eigenvalue of (\ref{eq28}).
\end{proof}

Then we get the following results from Lemma \ref{le5} and Theorem \ref{th2.1}.

\begin{theorem}\label{th3.1}
Suppose that $(A0)$, $(A3)$, (\ref{eq12}) and (\ref{eq13}) are satisfied.
\begin{itemize}
\item[(i)] If $\kappa\neq{0}$, then there exists a constant $\varepsilon>0$ such that for any $\lambda\in\Upsilon=(\lambda_{1}-\varepsilon,\lambda_{1})\cup(\lambda_{1},\lambda_{1}+\varepsilon)$,
model (\ref{eq3}) has exactly a spatially nonhomogeneous positive steady state solution, and $0$ is not the zero eigenvalue of the linearized equation (\ref{eq28}) of (\ref{eq3}).

\item[(ii)] If $\kappa={0}$
and $T_{u{u}}[\varphi_{1},\varphi_{1}]\neq{0}$ and $\varrho\nu<0$
\text{$($resp., $\varrho\nu>0$$)$}, then there exist a constant $\varepsilon>0$
such that for any $\lambda\in\Upsilon=(\lambda_{1},\lambda_{1}+\varepsilon)$
\text{$($resp., $\lambda\in\Upsilon=(\lambda_{1}-\varepsilon,\lambda_{1})$$)$}, model (\ref{eq3}) has exactly a spatially nonhomogeneous positive steady state solutions,
and $0$ is a zero eigenvalue of the linearized equation (\ref{eq28}) of (\ref{eq3}).
\end{itemize}
\end{theorem}

In the following, we are about to find the existence of purely imaginary eigenvalues of the linearized equation (\ref{eq28}) of (\ref{eq3}).
Note that for each $(\sigma,\lambda)\in{\mathbb{R_{+}}}\times\Upsilon$, $\mu=i\omega (\omega>0)$  is an eigenvalue of (\ref{eq28}) with eigenfunction $\psi$
if and only if there exists any $(\lambda,\omega,\sigma,\psi)\in\Upsilon\times{\mathbb{R_{+}}}\times{\mathbb{R_{+}}}\times{\mathcal{X}_{\mathbb{C}}\setminus{\{0\}}}$
such that $m(\lambda,i\omega,\sigma)\psi=0$. For later convenient analysis, unless otherwise specified, we always suppose that $\kappa\neq{0}$.

It follows from Lemma \ref{le4} that we let $\mu=i\omega=i\vartheta_{\lambda}^{*}\delta$, then we rewrite $m(\lambda,i\vartheta_{\lambda}^{*}\delta,\sigma)\psi=0$
as $M(\lambda,\delta,\theta)\psi=0$, where $M(\lambda,\delta,\theta): {\mathcal{X}_{\mathbb{C}}}\rightarrow{\mathcal{Y}_{\mathbb{C}}}$ is denoted by
\begin{small}
\begin{equation}\label{eq39}
M(\lambda,\delta,\theta)\psi=
\left(\begin{array}{c}
   \triangle{\psi}+d\nabla\cdot({\psi}\nabla{u_{\lambda}^{*}})+d\nabla\cdot({u_{\lambda}^{*}}\nabla{\psi})e^{-i\theta}
    +{\lambda}u_{\lambda}^{*}f_{u}(x,u_{\lambda}^{*})\psi+{\lambda}f(x,u_{\lambda}^{*})\psi-i\vartheta_{\lambda}^{*}\delta{\psi}\\
   \partial_{\overrightarrow{n}}\psi-{\lambda}r(x)g_{u}(u_{\lambda}^{*})\psi
\end{array}\right).
\end{equation}
\end{small}

Note that $M(\lambda_{1},\delta,\theta)=\mathcal{L}_{\lambda_{1}}$, and the decomposition of space $\mathcal{X}_{\mathbb{C}}=(\mathcal{N}(\mathcal{L}_{\lambda_{1}}))_{\mathbb{C}}\oplus({\mathcal{X}_{1}})_{\mathbb{C}}$,
and $\mathcal{N}(\mathcal{L}_{\lambda_{1}})=\text{span}\{\varphi_{1}\}$. After ignoring a scalar factor, $\psi$ is represented as
$\psi=\varphi_{1}+\vartheta_{\lambda}^{*}\beta$ with $\beta\in({\mathcal{X}_{1}})_{\mathbb{C}}$. Further, we get
$E(\beta,\delta,\theta,\lambda)=0$, where $E: ({\mathcal{X}_{1}})_{\mathbb{C}}\times{\mathbb{R}}\times[0,2\pi)\times\Upsilon\rightarrow{\mathcal{Y}_{\mathbb{C}}}$ is given by
\begin{eqnarray}\label{eq40}
E(\beta,\delta,\theta,\lambda)=\frac{1}{\vartheta_{\lambda}^{*}}M(\lambda,\delta,\theta)\varphi_{1}+M(\lambda,\delta,\theta)\beta,
\end{eqnarray}
in particular, we have $E(\beta,\delta,\theta,\lambda_{1})=P(\delta,\theta)+\mathcal{L}_{\lambda_{1}}\beta$, where $P(\delta,\theta)=\displaystyle\lim_{\lambda\rightarrow \lambda_{1}}\frac{1}{\vartheta_{\lambda}^{*}}M(\lambda,\delta,\theta)\varphi_{1}$,
$i.e.$,
\begin{eqnarray}\label{eq41}
\begin{split}
P(\delta,\theta)
=&
\left(\begin{array}{c}
   d\nabla\cdot({\varphi_{1}}\nabla{\varphi_{1}})+d\nabla\cdot({\varphi_{1}}\nabla{\varphi_{1}})e^{-i\theta}
    +2{\lambda_{1}}f_{u}(x,0){\varphi_{1}}^{2}-i\delta{\varphi_{1}}\\
   -{\lambda_{1}}r(x)g_{uu}(0){\varphi_{1}}^{2}
\end{array}\right)\\
&+
\left(\begin{array}{c}
   f(x,0){\varphi_{1}}\\
   -r(x)g_{u}(0){\varphi_{1}}
\end{array}\right)
\displaystyle\lim_{\lambda\rightarrow \lambda_{1}}\frac{\lambda-\lambda_{1}}{\vartheta_{\lambda}^{*}}
:=\left(\begin{array}{c}
   P_{1}(\delta,\theta)\\
   P_{2}(\delta,\theta)
\end{array}\right).
\end{split}
\end{eqnarray}

Throughout the remaining part of this section, we always suppose that
\begin{eqnarray}\label{eq42}
   4\kappa_{0}=\kappa\neq{0}  ~~\text{or}~~  \kappa(4\kappa_{0}- \kappa)>0,
\end{eqnarray}
which implies that $\kappa\neq{0}$, where
\begin{eqnarray*}
    \kappa_{0}=d\int_{\Omega}\varphi_{1}(x)\nabla\cdot({\varphi_{1}(x)}\nabla{\varphi_{1}(x)})dx.
\end{eqnarray*}

Define $F: {\mathbb{R}}\times[0,2\pi)\rightarrow{\mathbb{R}}$ as $F(\delta,\theta)=\langle{\Psi,P(\delta,\theta)}\rangle_{1}$. Further, we get
\begin{eqnarray}\label{eq43}
\begin{split}
F(\delta,\theta)=\int_{\Omega}\varphi_{1}P_{1}dx-\int_{\partial{\Omega}}\varphi_{1}P_{2}dS=-i\delta+\frac{\kappa}{2}-\kappa_{0}+e^{-i\theta}\kappa_{0}.
\end{split}
\end{eqnarray}

By the direct calculation, we obtain that $F(\delta_{*},\theta_{*})=0$, where
\begin{eqnarray}\label{eq44}
\delta_{*}=\frac{sgn\{{\vartheta_{\lambda}^{*}}\}\sqrt{\kappa(4\kappa_{0}- \kappa)}}{2},     \qquad  \theta_{*}=Arg\frac{-i\delta+\frac{\kappa}{2}-\kappa_{0}}{\kappa_{0}}.
\end{eqnarray}

In view of $F(\delta_{*},\theta_{*})=\langle{\Psi,P(\delta_{*},\theta_{*})}\rangle_{1}=0$, we have $P(\delta_{*},\theta_{*})\in\mathcal{R}(\mathcal{L}_{\lambda_{1}})$,
which means that there exists $\beta_{*}\in({\mathcal{X}_{1}})_{\mathbb{C}}\setminus{\{0\}}$ such that
$P(\delta_{*},\theta_{*})=-\mathcal{L}_{\lambda_{1}}\beta_{*}$, $i.e.$,
$E(\beta_{*},\delta_{*},\theta_{*},\lambda_{1})=0$. It follows from the implicit function theorem and the fact that $D_{(\beta,\delta,\theta)}E(\beta_{*},\delta_{*},\theta_{*},\lambda_{1})$ is a bijective that there exists a continuously differentiable mapping
$\lambda\mapsto{(\beta(\lambda),\delta(\lambda),\theta(\lambda))}$ from $\Upsilon$ to $({\mathcal{X}_{1}})_{\mathbb{C}}\times{\mathbb{R}}\times[0,2\pi)$ such that $(\beta(\lambda_{1}),\delta(\lambda_{1}),\theta(\lambda_{1}))=(\beta_{*},\delta_{*},\theta_{*})$, and $E(\beta(\lambda),\delta(\lambda),\theta(\lambda),\lambda)\equiv{0}$ for all $\lambda\in\Upsilon$.
Therefore, this means that $M(\lambda,\delta(\lambda),\theta(\lambda))\psi(\lambda)=0$ has a nontrivial solution $\psi(\lambda)=\varphi_{1}+\vartheta_{\lambda}^{*}\beta(\lambda)\in{\mathcal{X}_{\mathbb{C}}\setminus{\{0\}}}$
for each $\lambda\in\Upsilon$. Based on the above discussions, we have the following result.

\begin{col}\label{Cor1}
Suppose that $(A0)$, $(A3)$, (\ref{eq12}), (\ref{eq13}) and (\ref{eq42}) are satisfied,
then for each fixed $\lambda\in\Upsilon$, $\mu=i\omega(\omega>0)$  is a purely imaginary eigenvalue of (\ref{eq28}) with eigenfunction $\psi$
if and only if
\begin{eqnarray*}
\omega:=\omega(\lambda)=\vartheta_{\lambda}^{*}\delta(\lambda), ~~~~~~\sigma:=\sigma_{\lambda,n}=\frac{\theta(\lambda)+2n\pi}{\omega(\lambda)}, ~~n\in\mathbb{N}_{0},\\
\psi=c\psi(\lambda),  ~~~~~~~~~\psi(\lambda)=\varphi_{1}+\vartheta_{\lambda}^{*}\beta(\lambda), ~~~~~~~ \text{for any $\lambda\in\Upsilon$},
\end{eqnarray*}
where $c$ is a nonzero constant, $\delta(\lambda)$, $\theta(\lambda)$ and $\beta(\lambda)$ are denoted by the above discussions.
\end{col}

In the following, for the case that $(A0)$, $(A3)$, (\ref{eq12}), (\ref{eq13}) and (\ref{eq42}) are satisfied,
we shall investigate the linear stability and the existence of Hopf bifurcation of the positive spatially nonhomogeneous steady state of (\ref{eq28})
as $\lambda\in\Upsilon$. Now some estimates and necessary results are given for later analysis.

\begin{lemma}\label{le6}
Suppose that $(A0)$, $(A3)$, (\ref{eq12}), (\ref{eq13}) and (\ref{eq42}) are satisfied,
then $\displaystyle\lim_{\lambda\rightarrow{\lambda_{1}}}\Xi_{n}(\lambda)\neq{0}$ for $\lambda\in \Upsilon$ and $n\in\mathbb{N}_{0}$, where
\begin{eqnarray}\label{eq45}
\begin{split}
\Xi_{n}(\lambda):=\int_{\Omega}\psi(\lambda)\Big[\psi(\lambda)+\sigma_{\lambda,n}d\nabla\cdot(u_{\lambda}^{*}\nabla{\psi(\lambda)})e^{-i\theta(\lambda)}\Big]dx.
\end{split}
\end{eqnarray}
\end{lemma}
\begin{proof}
It follows from $\vartheta_{\lambda}^{*}(\lambda_{1})=0$ and
$\displaystyle\lim_{\lambda\rightarrow \lambda_{1}}\frac{u_{\lambda}^{*}}{\vartheta_{\lambda}^{*}}=\varphi_{1}$
and Corollary \ref{Cor1} that
\begin{eqnarray}\label{eq46}
\theta(\lambda)\mapsto\theta_{*},  \quad
\delta(\lambda)\mapsto\delta_{*},  \quad
\vartheta_{\lambda}^{*}\sigma_{\lambda,n}\mapsto\frac{\theta_{*}+2n\pi}{\delta_{*}}, \quad
\psi(\lambda)\mapsto{\varphi_{1}}, \quad
\text{in ${\mathcal{X}_{\mathbb{C}}}$ as $\lambda\mapsto{\lambda_{1}}$}.
\end{eqnarray}

Together with (\ref{eq46}), it follows from the dominated convergence theorem and $F(\delta_{*},\theta_{*})=0$ that
\begin{eqnarray}\label{eq47}
\begin{split}
\displaystyle\lim_{\lambda\rightarrow{\lambda_{1}}}\Xi_{n}(\lambda)
=&\int_{\Omega}{\varphi_{1}^{2}}dx+
\frac{\theta_{*}+2n\pi}{\delta_{*}}e^{-i\theta_{*}}d\int_{\Omega}\varphi_{1}\nabla\cdot({\varphi_{1}}\nabla{\varphi_{1}})dx\\
=&1+\frac{\theta_{*}+2n\pi}{\delta_{*}}e^{-i\theta_{*}}\kappa_{0}\\
=&1+\frac{\theta_{*}+2n\pi}{\delta_{*}}\Big[\kappa_{0}-\frac{\kappa}{2}+i\delta_{*}\Big]\neq{0},
\end{split}
\end{eqnarray}
since if ${\theta_{*}+2n\pi}=0$, then $\displaystyle\lim_{\lambda\rightarrow{\lambda_{1}}}Re\{\Xi_{n}(\lambda)\}\neq{0}$;
if ${\theta_{*}+2n\pi}\neq{0}$, then $\displaystyle\lim_{\lambda\rightarrow{\lambda_{1}}}Im\{\Xi_{n}(\lambda)\}\neq{0}$.
\end{proof}

Then we give the fact that the pure imaginary eigenvalue is simple.

\begin{lemma}\label{le7}
Suppose that $(A0)$, $(A3)$, (\ref{eq12}), (\ref{eq13}) and (\ref{eq42}) are satisfied,
then for $\lambda\in \Upsilon$, $\mu=i\omega(\lambda)$ is a simple eigenvalue of (\ref{eq28}) when $\sigma=\sigma_{\lambda,n}(n\in\mathbb{N}_{0})$
, where $\omega(\lambda)$ and $\sigma_{\lambda,n}$ are given by in Corollary \ref{Cor1}.
\end{lemma}
\begin{proof}
Differentiating the left side of (\ref{eq29}) with respect to $\mu$, we get
\begin{eqnarray}\label{eq48}
\begin{split}
-d\sigma\nabla\cdot({u_{\lambda}^{*}}\nabla{\psi})e^{-\mu{\sigma}}-\psi.
\end{split}
\end{eqnarray}

To be contrary, for $\lambda\in \Upsilon$,
$\mu=i\omega(\lambda)$ is at least a double eigenvalue of (\ref{eq28}) when $\sigma=\sigma_{\lambda,n}$.
Then together with (\ref{eq48}) and Corollary \ref{Cor1}, it follows that
$\mu=i\omega(\lambda)$ can satisfy the following equation
\begin{eqnarray}\label{eq49}
\begin{split}
-d\sigma_{\lambda,n}\nabla\cdot({u_{\lambda}^{*}}\nabla{\psi(\lambda)})e^{-i\theta(\lambda)}-\psi(\lambda)=0,
\end{split}
\end{eqnarray}
where $\sigma_{\lambda,n}$, $\psi(\lambda)$ and $\theta(\lambda)$ are defined in Corollary \ref{Cor1}. Multiplying two sides of (\ref{eq49}) by $-\psi(\lambda)$ and integrating over $\Omega$, we yield
\begin{eqnarray}\label{eq50}
\begin{split}
\Xi_{n}(\lambda):=\int_{\Omega}\psi(\lambda)\Big[\psi(\lambda)+\sigma_{\lambda,n}d\nabla\cdot(u_{\lambda}^{*}\nabla{\psi(\lambda)})e^{-i\theta(\lambda)}\Big]dx=0,
\end{split}
\end{eqnarray}
which is a contradiction with Lemma \ref{le6}. Hence, the proof is completed.
\end{proof}

This, together with the above analysis,  Lemma \ref{le7} and the implicit function theorem, implies that there exist a neighborhood
$B_{n}\times{D_{n}}\times{E_{n}}\subset{\mathbb{R}}\times{\mathbb{C}}\times{\mathcal{X}_{\mathbb{C}}}$ of $(\sigma_{\lambda,n},i\omega(\lambda),\psi(\lambda))$ and a continuously
differentiable function $(\mu,\psi):B_{n}\rightarrow{D_{n}}\times{E_{n}}$ such that for each $\sigma\in{B_{n}}$, $\mu(\sigma)\in{D_{n}}$ is the unique
eigenvalue of (\ref{eq28}) associated with the eigenvalue $\psi(\sigma)$, and
\begin{eqnarray}\label{eq51}
\begin{split}
\mu(\sigma_{\lambda,n})=&i\omega(\lambda),  \quad \psi(\sigma_{\lambda,n})=\psi(\lambda), \quad n\in{\mathbb{N}_{0}}\\
\triangle(\lambda,\mu(\sigma),\sigma)\psi(\sigma)
=&\triangle{\psi(\sigma)}+d\nabla\cdot({\psi(\sigma)}\nabla{u_{\lambda}^{*}})
+d\nabla\cdot({u_{\lambda}^{*}}\nabla{\psi(\sigma)})e^{-\mu(\sigma){\sigma}}\\
&+\lambda{u_{\lambda}^{*}}f_{u}(x,{u_{\lambda}^{*}}){\psi(\sigma)}+{\lambda}f(x,{u_{\lambda}^{*}}){\psi(\sigma)}-\mu(\sigma)\psi(\sigma)=0,
\end{split}
\end{eqnarray}
and for all $\psi(\sigma)\in\mathcal{X}_{\mathbb{C}}\setminus{\{0\}}$ satisfying $\partial_{\overrightarrow{n}}\psi(\sigma)={\lambda}r(x)g_{u}(u_{\lambda}^{*})\psi(\sigma)$ on $\partial{\Omega}$.

Now we shall verify the following transversality condition by some skillful calculations.

\begin{lemma}\label{le8}
Suppose that $(A0)$, $(A3)$, (\ref{eq12}), (\ref{eq13}) and (\ref{eq42}) are satisfied. For $\lambda\in \Upsilon$, we have
\begin{equation}\label{eq52}
    \frac{dRe\{\mu(\sigma)\}}{d\sigma}\bigg|_{\sigma=\sigma_{\lambda,n}}>0,   \quad  n\in{\mathbb{N}_{0}}.
\end{equation}
\end{lemma}
\begin{proof}
Differentiating (\ref{eq51}) with respect to $\sigma$ at $\sigma=\sigma_{\lambda,n}$, we obtain
\begin{eqnarray}\label{eq53}
\begin{split}
0=
&\triangle(\lambda,i\omega(\lambda),\sigma_{\lambda,n})\frac{d\psi(\sigma_{\lambda,n})}{d\sigma}
-i\omega(\lambda)d\nabla\cdot({u_{\lambda}^{*}}\nabla{\psi(\lambda)})e^{-i\theta(\lambda)}\\
&-\frac{d\mu(\sigma_{\lambda,n})}{d\tau}\Big[\psi(\lambda)+\sigma_{\lambda,n}d\nabla\cdot({u_{\lambda}^{*}}\nabla{\psi(\lambda)})e^{-i\theta(\lambda)}\Big].
\end{split}
\end{eqnarray}

Multiplying two sides of (\ref{eq53}) by $\psi(\lambda)$ and integrating over $\Omega$, it follows from (\ref{eq45}) that
\begin{eqnarray}\label{eq54}
\begin{split}
\frac{d\mu(\sigma_{\lambda,n})}{d\sigma}\Xi_{n}(\lambda)=
-i\omega(\lambda)de^{-i\theta(\lambda)}\int_{\Omega}\psi(\lambda)\nabla\cdot({u_{\lambda}^{*}}\nabla{\psi(\lambda)})dx,
\end{split}
\end{eqnarray}
which yields
\begin{eqnarray}\label{eq55}
\begin{split}
\frac{d\mu(\sigma_{\lambda,n})}{d\sigma}
&=\frac{-i\omega(\lambda){I_{0}}\int_{\Omega}\psi^{2}(\lambda)dx-i\omega(\lambda){\sigma_{\lambda,n}}|I_{0}|^{2}}{|\Xi_{n}(\lambda)|^{2}},
\end{split}
\end{eqnarray}
where
\begin{eqnarray}\label{eq56}
\begin{split}
I_{0}=de^{-i\theta(\lambda)}\int_{\Omega}\psi(\lambda)\nabla\cdot({u_{\lambda}^{*}}\nabla{\psi(\lambda)})dx.
\end{split}
\end{eqnarray}

This, together with the dominated convergence theorem and (\ref{eq46}), implies that
there is the following result
\begin{eqnarray}\label{eq57}
\begin{split}
\displaystyle\lim_{\lambda\rightarrow{\lambda_{1}}}\frac{1}{(\vartheta_{\lambda}^{*})^{2}}\frac{dRe\{\mu(\sigma_{\lambda,n})\}}{d\sigma}
&=\frac{-\delta_{*}\kappa_{0}{Re}\big\{ie^{-i\theta_{*}}\big\}
-(\theta_{*}+2n\pi)\kappa_{0}^{2}{Re}\big\{i\big\}}
{\displaystyle\lim_{\lambda\rightarrow{\lambda_{1}}}|\Xi_{n}(\lambda)|^{2}}\\
&=\frac{-\delta_{*}\kappa_{0}\sin{\theta_{*}}}
{\displaystyle\lim_{\lambda\rightarrow{\lambda_{1}}}|\Xi_{n}(\lambda)|^{2}}
=\frac{\delta_{*}^{2}}
{\displaystyle\lim_{\lambda\rightarrow{\lambda_{1}}}|\Xi_{n}(\lambda)|^{2}}>0,
\end{split}
\end{eqnarray}
since it follows from $F(\delta_{*},\theta_{*})=0$ that we get $\delta_{*}=-\kappa_{0}\sin{\theta_{*}}$.

\end{proof}

\begin{remark}\label{remark2}
This, together with Corollary \ref{Cor1} and Lemmas \ref{le3}-\ref{le8},
implies that around the steady state solution ${u_{\lambda}^{*}}$ established in Theorem \ref{th2.1},
the Hopf bifurcation can occur at $\sigma=\sigma_{\lambda,n}$ for $(\lambda,n)\in\Upsilon\times{\mathbb{N}_{0}}$ with $\lambda$ sufficiently close to ${\lambda_{1}}$.
In other words, a branch of periodic orbits of (\ref{eq3}) emerges from $(\sigma,u)=(\sigma_{\lambda,n},{u_{\lambda}^{*}})$.
\end{remark}

Now let
\begin{eqnarray}\label{eq58}
\begin{cases}
     \kappa_{1}=\lambda_{1}\int_{\Omega}\varphi_{1}^{3}(x)f_{u}(x,0)dx,     \\
    \kappa_{2}=\lambda_{1}g_{uu}(0)\int_{\partial{\Omega}}\varphi_{1}^{3}(x)r(x)dS,
\end{cases}
\end{eqnarray}
then we yield $\kappa=2\kappa_{0}+2\kappa_{1}+\kappa_{2}$, which means that (\ref{eq42}) can be rewritten as
\begin{eqnarray}\label{eq59}
\begin{cases}
   4\kappa_{0}=\kappa\neq{0}  \Rightarrow 2\kappa_{0}={2\kappa_{1}+\kappa_{2}}\neq-2\kappa_{0}, ~~\text{or}\\
   \kappa(4\kappa_{0}- \kappa)>0 \Rightarrow   4\kappa_{0}^{2}-(2\kappa_{1}+\kappa_{2})^{2}>0.
 \end{cases}
\end{eqnarray}

\begin{remark}\label{remark3}
Condition (\ref{eq42})($i.e.$, (\ref{eq59})) implies that no matter how the memory delay $\sigma$ changes, the Hopf bifurcation cannot occur in model (\ref{eq3})
near $u_{\lambda}^{*}$ if the memory reaction term is weaker than the interaction of the interior reaction term and the boundary reaction term. Whereas,
if the memory reaction term is stronger than the interaction of the interior reaction term and the boundary reaction term, it is the memory delay $\sigma$
that determines the existence of Hopf bifurcation in model (\ref{eq3}) near $u_{\lambda}^{*}$.
\end{remark}

For $\lambda\in \Upsilon$, we next consider the existence of purely imaginary eigenvalues of (\ref{eq28}) when $\sigma=0$ and analyse the stability of
the steady state solution ${u_{\lambda}^{*}}$ established in Theorem \ref{th2.1} by regarding the memory delay $\sigma$ as a parameter.

\begin{lemma}\label{le9}
Suppose that $(A0)$, $(A3)$, (\ref{eq12}), (\ref{eq13}) and $\kappa\neq{0}$ are satisfied, then for $\lambda\in \Upsilon$ and $\sigma=0$, the eigenvalues of (\ref{eq28}) have only eigenvalues with negative real parts if $\varrho(\lambda-\lambda_{1})>0$, while the eigenvalues of (\ref{eq28}) have at least one eigenvalue with a positive real part if $\varrho(\lambda-\lambda_{1})<0$.
\end{lemma}
\begin{proof}
At first, we obtain that the eigenvalues of (\ref{eq28}) have no purely imaginary eigenvalues if $\sigma=0$ and $\kappa\neq{0}$.
In fact, if  the eigenvalues of (\ref{eq28}) have purely imaginary eigenvalues, then it follows from the above similar analysis that we can get
$F(\delta,0)=-i\delta+\frac{\kappa}{2}\neq{0}$, which is a contradiction.

For each $j\in\mathbb{N}$, let $\big\{{\lambda_{j}^{*}}\big\}_{j=1}^{\infty}\subseteq\Upsilon$ be a sequence, ${\lambda_{j}^{*}}\rightarrow{\lambda_{1}}$ as $j\rightarrow\infty$ such that
there exist an eigenvalue $\mu_{j}$ of (\ref{eq28}), and  a sequence $\big\{{\psi_{j}^{*}}\big\}_{j=1}^{\infty}\subseteq\mathcal{X}_{\mathbb{C}}\setminus{\{0\}}$, with ${\psi_{j}^{*}}\rightarrow{\varphi_{1}}$ as $j\rightarrow\infty$,  such that $m({\lambda_{j}^{*}},\mu_{j},0){\psi_{j}^{*}}=0$, where
\begin{small}
\begin{eqnarray}\label{eq60}
\begin{split}
   &0=\triangle{\psi_{j}^{*}}+d\nabla\cdot({\psi_{j}^{*}}\nabla{u_{\lambda_{j}^{*}}^{*}})+d\nabla\cdot({u_{\lambda_{j}^{*}}^{*}}\nabla{\psi_{j}^{*}})
    +{\lambda_{j}^{*}}u_{\lambda_{j}^{*}}^{*}f_{u}(x,u_{\lambda_{j}^{*}}^{*}){\psi_{j}^{*}}+{\lambda_{j}^{*}}f(x,u_{\lambda_{j}^{*}}^{*}){\psi_{j}^{*}}-\mu_{j}{\psi_{j}^{*}}, ~\text{in $\Omega$},\\
   &0=\partial_{\overrightarrow{n}}{\psi_{j}^{*}}-{\lambda_{j}^{*}}r(x)g_{u}(u_{\lambda_{j}^{*}}^{*}){\psi_{j}^{*}}, ~~~~~~~~~~~~~~~~~~~~~~~~~~~~~~~~~~~~~~~~~~~~~~~~~~~~~~~~~~~~~~~~~~~~~~~~~~~~~~\text{on $\partial{\Omega}$}.
\end{split}
\end{eqnarray}
\end{small}

This, together with
$\displaystyle\lim_{j\rightarrow {\infty}}\frac{u_{\lambda_{j}^{*}}^{*}}{\lambda_{j}^{*}-{\lambda}_{1}}
=\varphi_{1}\displaystyle\lim_{j\rightarrow {\infty}}\frac{\vartheta_{\lambda_{j}^{*}}^{*}}{\lambda_{j}^{*}-{\lambda}_{1}}
=-\frac{2\varrho\varphi_{1}}{\kappa}$
and $\langle{\Psi,m({\lambda_{j}^{*}},\mu_{j},0){\psi_{j}^{*}}}\rangle_{1}=0$, implies that
$\displaystyle\lim_{j\rightarrow {\infty}}\mu_{j}=0$ and
\begin{eqnarray}\label{eq61}
\begin{split}
\displaystyle\lim_{j\rightarrow {\infty}}\frac{\mu_{j}}{\lambda_{j}^{*}-{\lambda}_{1}}
=-\frac{2\varrho}{\kappa}\big[2\kappa_{0}+2\kappa_{1}+\kappa_{2}\big]+\varrho=-\varrho,
\end{split}
\end{eqnarray}
which means that $sign \Big(Re\big\{\mu_{j}\big\}\Big)=sign \Big(\varrho(\lambda_{1}-\lambda_{j}^{*})\Big)$ for sufficiently large $j$. Thus, the proof is completed.
\end{proof}

\begin{remark}\label{remark4}
 Suppose that $(A0)$, $(A3)$, (\ref{eq12}), (\ref{eq13}) and $\kappa\neq{0}$ are satisfied, Lemma \ref{le9} implies that if $\sigma=0$, the eigenvalues of (\ref{eq28})
 with $\lambda\in \Upsilon$ satisfying $\varrho(\lambda-\lambda_{1})>0$ have negative real parts. Namely, the steady state solution ${u_{\lambda}^{*}}$ established
 in Theorem \ref{th2.1} is linearly stable when $\sigma=0$ and $\lambda\in \Upsilon$ satisfying $\varrho(\lambda-\lambda_{1})>0$.
 However, the above analysis means that the presence of the memory delay $\sigma$ may lead to some nonlinear oscillations and cause the complex dynamical  behaviors.
 Hence, the memory delay $\sigma$ could be considered as oscillatory response of (\ref{eq3}) near $u_{\lambda}^{*}$.
\end{remark}

Thus, it follows from the above discussions that we get following results for the steady state solution ${u_{\lambda}^{*}}$ established in Theorem \ref{th2.1}.

\begin{theorem}\label{th3.2}
Suppose that $(A0)$, $(A3)$, (\ref{eq12}) and (\ref{eq13}) are satisfied.
\begin{itemize}
\item[(i)] If $\kappa\neq{0}$, then for each $(\sigma,\lambda)\in{\mathbb{R}}\setminus{\mathbb{R_{-}}}\times\Upsilon$ satisfying $\varrho(\lambda-\lambda_{1})<0$
and $\big|\lambda-\lambda_{1}\big|\ll{1}$,
the eigenvalues of (\ref{eq28}) have at least one eigenvalue with a positive real part, and the steady state solution ${u_{\lambda}^{*}}$ established in Theorem \ref{th2.1} is unstable.

\item[(ii)] If $\kappa(4\kappa_{0}- \kappa)<0$, which implies that $\kappa\neq{0}$, then then for each $(\sigma,\lambda)\in{\mathbb{R}}\setminus{\mathbb{R_{-}}}\times\Upsilon$ satisfying $\varrho(\lambda-\lambda_{1})>0$ and $\big|\lambda-\lambda_{1}\big|\ll{1}$,
    the eigenvalues of (\ref{eq28}) have only eigenvalues with negative real parts, and the steady state solution ${u_{\lambda}^{*}}$ established in Theorem \ref{th2.1} is linearly stable.

\item[(iii)] Under the condition (\ref{eq42}), then for each $(\sigma,\lambda)\in{\mathbb{R}}\setminus{\mathbb{R_{-}}}\times\Upsilon$ satisfying $\varrho(\lambda-\lambda_{1})>0$ and $\big|\lambda-\lambda_{1}\big|\ll{1}$,
    the eigenvalues of (\ref{eq28}) have only eigenvalues with negative real parts when $\sigma\in[0,\sigma_{\lambda,0})$,
    and precisely $2(n+1)$ eigenvalues with positive real parts when $\sigma\in(\sigma_{\lambda,n},\sigma_{\lambda,n+1}]$, $n\in{\mathbb{N}_{0}}$.
    Therefore, the steady state solution ${u_{\lambda}^{*}}$ established in Theorem \ref{th2.1} is linearly stable when $\sigma\in[0,\sigma_{\lambda,0})$,
    and is unstable when $\sigma_{\lambda,0}<\sigma$, and undergoes Hopf bifurcation at $\sigma=\sigma_{\lambda,0}$.

\end{itemize}
\end{theorem}

\section{Stability of steady states in Theorem \ref{th2.2}}
From Theorem \ref{th2.2}, for any given $u_{1}\in\mathbb{R_{+}}$ satisfying $(A1)$ and $(A2)$, model (\ref{eq3}) near $(u,\lambda)=(u_{1},0)$ has precisely two curves of
steady state solutions: one is the spatially homogeneous positive steady state solution $(u,\lambda)=(u_{1},0)$
from the curve $\Gamma_{u_{1}}$, and the other is the spatially nonhomogeneous positive steady state solution  $(u,\lambda)=(u_{01}(s),\lambda_{01}(s))$ with
$s\in(-\varepsilon_{1}, \varepsilon_{1}), \varepsilon_{1}\in\mathbb{R_{+}}$ from the curve $\Gamma_{\Xi}$, where
$u_{01}(s)=u_{1}+\eta_{1}s+s\zeta_{1}(s)$ and $\lambda_{01}(s)=s+s\xi_{1}(s)$ with $\zeta_{1}(0)=\zeta_{1}^{'}(0)=\xi_{1}(0)=\xi_{1}^{'}(0)=0$ can satisfy
\begin{eqnarray}\label{eq62}
\left \langle {\bar{l},T(u_{01}(s),\lambda_{01}(s))} \right \rangle_{2}:=\int_{\Omega}T_{1}(u_{01}(s),\lambda_{01}(s))dx-(1+du_{1})\int_{\partial{\Omega}}T_{2}(u_{01}(s),\lambda_{01}(s))dS=0,
\end{eqnarray}
where $s\in(-\varepsilon_{1}, \varepsilon_{1})$. In particular, we get $(u_{01}(s),\lambda_{01}(s))=(u_{1},0)$ as $s=0$. In the following, we will investigate the stability of the
bifurcated steady state $u=u_{01}(s)$ with $\lambda=\lambda_{01}(s)$ for $s\in\tilde{\Upsilon}=(-\varepsilon_{1},0)\cup(0, \varepsilon_{1})$, when $\Omega$ is a bounded open set in $\mathbb{R}$.
To analyse the stability of model (\ref{eq3}) near $(u,\lambda)=(u_{01}(s),\lambda_{01}(s))$, we consider the following linearized equation of (\ref{eq3})
at $u=u_{01}(s)$ with $\lambda=\lambda_{01}(s)$
\begin{equation}\label{eq63}
 \begin{cases}
    \frac{\partial{v}}{\partial{t}}=\triangle{v}+d\nabla\cdot({v}\nabla{u_{01}(s)})+d\nabla\cdot({u_{01}(s)}\nabla{v_{\sigma}})\\
             ~~~~~~+{\lambda_{01}(s)}u_{01}(s)f_{u}(x,u_{01}(s))v+{\lambda_{01}(s)}f(x,u_{01}(s))v,            ~~\text{$x \in \Omega, t >0 $},\\
     \partial_{\overrightarrow{n}}v={\lambda_{01}(s)}r(x)g_{u}(u_{01}(s))v,               ~~~~~~~~~~~~~~~~~~~~~~~~~~~~~~~~~~~\text{$x \in \partial{\Omega}, t >0 $},
  \end{cases}
\end{equation}
where $v=v(x,t)$ and $v_{\sigma}=v(x,t-\sigma)$. Then using the similar arguments in Section $3$, we are looking for $(\mu,\psi)\in{\mathbb{C}}\times\mathcal{X}_{\mathbb{C}}\setminus{\{0\}}$
with any $s\in\tilde{\Upsilon}$ and $\sigma\in{\mathbb{R}}\setminus{\mathbb{R_{-}}}$,
such that
\begin{eqnarray}\label{eq64}
\tilde{\triangle}(s,\mu,\sigma)\psi(x)=0,          &\text{$x \in \Omega$},
\end{eqnarray}
where
\begin{eqnarray*}
\begin{split}
\tilde{\triangle}(s,\mu,\sigma)\psi=&
   \triangle{\psi}+d\nabla\cdot({\psi}\nabla{u_{01}(s)})+d\nabla\cdot({u_{01}(s)}\nabla{\psi})e^{-\mu\sigma}\\
    &+{\lambda_{01}(s)}u_{01}(s)f_{u}(x,u_{01}(s))\psi+{\lambda_{01}(s)}f(x,u_{01}(s))\psi-\mu{\psi},
\end{split}
\end{eqnarray*}
and for all $\psi\in\mathcal{X}_{\mathbb{C}}\setminus{\{0\}}$ satisfying $\partial_{\overrightarrow{n}}\psi(x)={\lambda_{01}(s)}r(x)g_{u}(u_{01}(s))\psi(x)$ on
$x\in \partial{\Omega}$.

If any $(s,\mu,\sigma,\psi)\in\tilde{\Upsilon}\times{\mathbb{C}}\times{{\mathbb{R}}\setminus{\mathbb{R_{-}}}}\times{\mathcal{X}_{\mathbb{C}}\setminus{\{0\}}}$ can solve (\ref{eq64})
and  $\partial_{\overrightarrow{n}}\psi={\lambda_{01}(s)}r(x)g_{u}(u_{01}(s))$ $\psi$ holds on $\partial{\Omega}$, $\mu\in{\mathbb{C}}$ can be regarded as an eigenvalue of (\ref{eq63})
with eigenfunction $\psi$, then without loss of generality, we assume that $\|{\psi}\|_{\mathcal{Y}_{\mathbb{C}}}=1$ for simplicity. Consequently, we respectively give two
prior estimates for the eigenfunction $\psi_{s}$ and eigenvalue $\mu_{s}$ with $s\in\tilde{\Upsilon}$ by using
arguments as in Section $3$. First, we get a priori estimate for the eigenfunction.

\begin{lemma}\label{le10}
Suppose that $(A0)$, $(A1)$, $(A2)$ and $(\tilde{A3})$ are satisfied, where
\begin{eqnarray*}
(\tilde{A3}) ~~|d|<\tilde{d}_{*}:=\frac{1}{\displaystyle\max_{s\in\tilde{\Upsilon}}\max_{x\in\bar{\Omega}}\big\{{u_{01}(s)}\big\}},~~
~~~~~~~~~~~~~~~~~~~~~~~~~~~~~~~~~~~~~~~~~~~~~~~~~~~~~~~~~~~~~~~~~~~~~~~~~
\end{eqnarray*}
then there exists a constant $\mathbf{\tilde{C}}$, such that for any
$(s,\mu_{s},\sigma_{s},\psi_{s})\in\tilde{\Upsilon}\times{\mathbb{C}}\times{\mathbb{R}}\setminus{\mathbb{R_{-}}}\times{\mathcal{X}_{\mathbb{C}}\setminus{\{0\}}}$
with $Re\{\mu_{s}\}\geq 0$ solving (\ref{eq64}) and satisfying $\partial_{\overrightarrow{n}}\psi_{s}={\lambda_{01}(s)}r(x)g_{u}(u_{01}(s))\psi_{s}$ on $\partial{\Omega}$, and
\begin{eqnarray*}
  \big\|{\nabla\psi_{s}}\big\|_{\mathcal{Y}_{\mathbb{C}}}^{2}\leq{\mathbf{\tilde{C}}}.
  \end{eqnarray*}
\end{lemma}

Then a priori estimate for the eigenvalue can be obtained in the following lemma.
\begin{lemma}\label{le11}
Suppose that $(A0)$, $(A1)$, $(A2)$ and $(\tilde{A3})$ are satisfied. If $(s,\mu_{s},\sigma_{s},\psi_{s})\in\tilde{\Upsilon}\times{\mathbb{C}}\times{\mathbb{R}}\setminus{\mathbb{R_{-}}}\times{\mathcal{X}_{\mathbb{C}}\setminus{\{0\}}}$
with $Re\{\mu_{s}\}\geq 0$ can solve (\ref{eq64}) and $\partial_{\overrightarrow{n}}\psi_{s}={\lambda_{01}(s)}r(x)g_{u}(u_{01}(s))\psi_{s}$ holds on $\partial{\Omega}$,
then $\Big|\frac{\mu_{s}}{s}\Big|$ is bounded for $s\in\tilde{\Upsilon}$.
\end{lemma}

In what follows, we give some results on the existence of the zero eigenvalue of (\ref{eq63}).
For convenience, define
\begin{equation}\label{eq65}
\tilde{m}(s,\mu,\sigma)\psi:=
\left(\begin{array}{c}
   \tilde{\triangle}(s,\mu,\sigma)\psi\\
   \partial_{\overrightarrow{n}}\psi-{\lambda_{01}(s)}r(x)g_{u}(u_{01}(s))\psi
\end{array}\right).
\end{equation}
As $(A2)$ holds, it is not difficult to find that $\left \langle {\bar{l},T_{uu}(u_{1},0)[1,\psi_{*}]} \right \rangle_{2}=0$, which implies
$\left \langle {\bar{l},T_{\lambda{u}}(u_{1},0)[1]} \right \rangle_{2}\neq{0}$; and  $\left \langle {\bar{l},T_{\lambda{u}}(u_{1},0)[1]} \right \rangle_{2}=0$, which means
$\left \langle {\bar{l},T_{uu}(u_{1},0)[1,\psi_{*}]} \right \rangle_{2}\neq{0}$.

\begin{lemma}\label{le12}
Suppose that $(A0)$, $(A1)$, $(A2)$, $(\tilde{A3})$ and $d\triangle{\eta_{1}}=0$ are satisfied.
\begin{itemize}
\item[(i)] For $\forall(\sigma,s)\in{\mathbb{R}}\setminus{\mathbb{R_{-}}}\times\tilde{\Upsilon}$, if $\left \langle {\bar{l},T_{uu}(u_{1},0)[1,\psi_{*}]} \right \rangle_{2}=0$,
then $0$ is not a zero eigenvalue of (\ref{eq63}).

\item[(ii)] For $\forall(\sigma,s)\in{\mathbb{R}}\setminus{\mathbb{R_{-}}}\times\tilde{\Upsilon}$, if $\left \langle {\bar{l},T_{\lambda{u}}(u_{1},0)[1]} \right \rangle_{2}=0$ and
$T_{\lambda{u}}(u_{1},0)[1]\neq{0}$, then $0$ is a zero eigenvalue of (\ref{eq63}).
\end{itemize}
\end{lemma}
\begin{proof}
$(i)$ Suppose  that $0$ is the zero eigenvalue of (\ref{eq63}), then there exists some $\psi\in{\mathcal{X}_{\mathbb{C}}\setminus{\{0\}}}$
with $(\sigma,s)\in{\mathbb{R}}\setminus{\mathbb{R_{-}}}\times\tilde{\Upsilon}$
such that $\tilde{m}(s,0,\sigma)\psi=0$. This, together with $\tilde{m}(0,0,\sigma)=T_{u}(u_{1},0)$, the decomposition of space $\mathcal{X}_{\mathbb{C}}=(\mathcal{N}(T_{u}(u_{1},0)))_{\mathbb{C}}\oplus({\mathcal{X}_{2}})_{\mathbb{C}}$
and $\mathcal{N}(T_{u}(u_{1},0))=\text{span}\{1\}$,  implies that $\psi$ takes the following after ignoring a scalar factor
\begin{eqnarray}\label{eq66}
\begin{cases}
\psi:=\psi_{s}=\tilde{\alpha}_{s}+s\tilde{\beta}_{s}, ~~\tilde{\beta}_{s}\in({\mathcal{X}_{2}})_{\mathbb{C}},
~~\tilde{\alpha}_{s}\geq 0,\\
\|\psi_{s}\|^{2}_{\mathcal{Y}_{\mathbb{C}}}=
\tilde{\alpha}_{s}^{2}\big|{\Omega}\big|+s^{2}\|{\tilde{\beta}_{s}}\|^{2}_{\mathcal{Y}_{\mathbb{C}}}
=\big|{\Omega}\big|,
\end{cases}
\end{eqnarray}
then $\tilde{m}(s,0,\sigma)\psi_{s}=
\tilde{\alpha}_{s}\tilde{m}(s,0,\sigma)+s\tilde{m}(s,0,\sigma)\tilde{\beta}_{s}=0$. Then we can find that
\begin{eqnarray}\label{eq67}
\displaystyle\lim_{s\rightarrow {0}}\frac{1}{s}\tilde{m}(s,0,\sigma)\cdot{1}=T_{\lambda{u}}(u_{1},0)[1]+(d\triangle{\eta_{1}},0)^{T}=T_{\lambda{u}}(u_{1},0)[1],
\end{eqnarray}
thus, we have $\tilde{\alpha}_{0}T_{\lambda{u}}(u_{1},0)[1]+T_{u}(u_{1},0)\tilde{\beta}_{0}=0$.
It follows from $\left \langle {\bar{l},T_{uu}(u_{1},0)[1,\psi_{*}]} \right \rangle_{2}=0$
and $(A2)$ that we get $\left \langle {\bar{l},T_{\lambda{u}}(u_{1},0)[1]} \right \rangle_{2}\neq{0}$,
which means that $T_{\lambda{u}}(u_{1},0)[1]\notin\mathcal{R}(T_{u}(u_{1},0))$. Further, we obtain $\tilde{\alpha}_{0}=0$.
Therefore, it follows,  from the fact that $\tilde{\beta}_{0}\in({\mathcal{X}_{2}})_{\mathbb{C}}$ and $T_{u}(u_{1},0)$ is invertible
as it is restricted in $({\mathcal{X}_{2}})_{\mathbb{C}}$,  that  $\tilde{\beta}_{0}\equiv{0}$. This implies that
$\displaystyle\lim_{s\rightarrow {0}}\tilde{m}(s,0,\sigma)\psi_{s}=0$ has only the zero solution, which contradicts
$\psi_{s}\in{\mathcal{X}_{\mathbb{C}}\setminus{\{0\}}}$. Hence, $0$ is not a zero eigenvalue of (\ref{eq63}).

$(ii)$  In order to verify that $0$ is a zero eigenvalue of (\ref{eq63}), it suffices to find a solution of $(s,\sigma,\psi)\in\tilde{\Upsilon}\times{\mathbb{R}}\setminus{\mathbb{R_{-}}}\times{\mathcal{X}_{\mathbb{C}}\setminus{\{0\}}}$ to the equation $\tilde{m}(s,0,\sigma)\psi=0$.
Substituting $\psi=1+s\tilde{\beta}$ with $\tilde{\beta}\in({\mathcal{X}_{2}})_{\mathbb{C}}$ into $\tilde{m}(s,0,\sigma)\psi=0$, we yield
\begin{eqnarray}\label{eq68}
\tilde{H}(\tilde{\beta},s):=\frac{1}{s}\tilde{m}(s,0,\sigma)\cdot{1}+\tilde{m}(s,0,\sigma)\tilde{\beta}=0,
\end{eqnarray}
which means that $\tilde{H}(\tilde{\beta},0)=T_{\lambda{u}}(u_{1},0)[1]+T_{u}(u_{1},0)\tilde{\beta}$.
If $\left \langle {\bar{l},T_{\lambda{u}}(u_{1},0)[1]} \right \rangle_{2}=0$, then $T_{\lambda{u}}(u_{1},0)[1]\in\mathcal{R}(T_{u}(u_{1},0))$.
This, together with $T_{\lambda{u}}(u_{1},0)[1]\neq{0}$, implies that there exists $\tilde{\beta}^{*}\in({\mathcal{X}_{2}})_{\mathbb{C}}\setminus{\{0\}}$ such that $T_{\lambda{u}}(u_{1},0)[1]=-T_{u}(u_{1},0)\tilde{\beta}^{*}$,
$i.e.$, $\tilde{H}(\tilde{\beta}^{*},0)=0$.
Note that $T_{\lambda{u}}(u_{1},0)[1]\neq{0}$ and $\tilde{H}_{\tilde{\beta}}(\tilde{\beta}^{*},0)=T_{u}(u_{1},0)$ is invertible
as it is restricted in $({\mathcal{X}_{2}})_{\mathbb{C}}$, then it follows from the implicit function theorem that we have a continuously differentiable mapping
$s\mapsto{\tilde{\beta}_{s}}$ from $\tilde{\Upsilon}$ to $({\mathcal{X}_{2}})_{\mathbb{C}}$ such that $\tilde{\beta}_{0}=\tilde{\beta}^{*}$, and $\tilde{H}(\tilde{\beta}_{s},s)\equiv{0}$. This implies that $\tilde{m}(s,0,\sigma)\psi=0$ has a nontrivial solution $\psi=1+s\tilde{\beta}_{s}\in{\mathcal{X}_{\mathbb{C}}\setminus{\{0\}}}$
for each $(\sigma,s)\in{\mathbb{R}}\setminus{\mathbb{R_{-}}}\times\tilde{\Upsilon}$. Hence, $0$ is a zero eigenvalue of (\ref{eq63}).
\end{proof}

Then we obtain the following results from Lemma \ref{le12} and Theorem \ref{th2.2}.

\begin{theorem}\label{th4.1}
Suppose that $(A0)$, $(A1)$, $(A2)$, $(\tilde{A3})$ and $d\triangle{\eta_{1}}=0$ are satisfied.
\begin{itemize}
\item[(i)] If $\left \langle {\bar{l},T_{uu}(u_{1},0)[1,\psi_{*}]} \right \rangle_{2}=0$, then there exist a constant $\varepsilon_{1}\in\mathbb{R_{+}}$ such that for any $s\in\tilde{\Upsilon}=(-\varepsilon_{1},0)\cup(0, \varepsilon_{1})$,
model (\ref{eq3}) has exactly a spatially nontrivial steady state solution, and $0$ is not a zero eigenvalue of the linearized equation (\ref{eq63}) of (\ref{eq3}).

\item[(ii)] If $\left \langle {\bar{l},T_{\lambda{u}}(u_{1},0)[1]} \right \rangle_{2}=0$ and
$T_{\lambda{u}}(u_{1},0)[1]\neq{0}$, then there exist a constant $\varepsilon_{1}\in\mathbb{R_{+}}$ such that for any $s\in\tilde{\Upsilon}=(-\varepsilon_{1},0)\cup(0, \varepsilon_{1})$,
model (\ref{eq3}) has exactly a spatially nontrivial steady state solution, and $0$ is a zero eigenvalue of the linearized equation (\ref{eq63}) of (\ref{eq3}).
\end{itemize}
\end{theorem}

Next, we are about to consider the existence of purely imaginary eigenvalues of the linearized equation (\ref{eq63}) of (\ref{eq3}).
Since $(\sigma,s)\in{\mathbb{R}}\setminus{\mathbb{R_{-}}}\times\tilde{\Upsilon}$, $\mu=i\tilde{\omega} (\tilde{\omega}>0)$  is an eigenvalue of (\ref{eq63}) with eigenfunction $\psi$
if and only if there exists any $(s,\tilde{\omega},\sigma,\psi)\in\tilde{\Upsilon}\times{\mathbb{R_{+}}}\times{\mathbb{R}}\setminus{\mathbb{R_{-}}}\times{\mathcal{X}_{\mathbb{C}}\setminus{\{0\}}}$
such that $\tilde{m}(s,i\tilde{\omega},\sigma)\psi=0$. To simplify, unless otherwise specified, we always suppose that
$(A0)$, $(A1)$, $(A2)$, $(\tilde{A3})$, $d\triangle{\eta_{1}}=0$ and $\left \langle {\bar{l},T_{uu}(u_{1},0)[1,\psi_{*}]} \right \rangle_{2}=0$ are satisfied.

In the view of Lemma \ref{le11}, let $\mu=i\tilde{\omega}=is\tilde{\delta}$, then we rewrite $\tilde{m}(s,is\tilde{\delta},\sigma)\psi=0$
as $\tilde{M}(s,\tilde{\delta},\tilde{\theta})\psi=0$, where $\tilde{M}(s,\tilde{\delta},\tilde{\theta}): {\mathcal{X}_{\mathbb{C}}}\rightarrow{\mathcal{Y}_{\mathbb{C}}}$ is given by
\begin{small}
\begin{equation}\label{eq69}
\tilde{M}(s,\tilde{\delta},\tilde{\theta})\psi=
\left(\begin{array}{c}
   \triangle{\psi}+d\nabla\cdot({\psi}\nabla{u_{01}(s)})+d\nabla\cdot({u_{01}(s)}\nabla{\psi})e^{-i\tilde{\theta}}
   +{\lambda_{01}(s)}u_{01}(s)f_{u}(x,u_{01}(s))\psi\\
    ~~~~~~~~~~~~~~~+{\lambda_{01}(s)}f(x,u_{01}(s))\psi-is\tilde{\delta}{\psi}\\

   \partial_{\overrightarrow{n}}\psi-{\lambda_{01}(s)}r(x)g_{u}(u_{01}(s))\psi
\end{array}\right).
\end{equation}
\end{small}

This, together with $\tilde{M}(0,\tilde{\delta},\tilde{\theta})=T_{u}(u_{1},0)+(du_{1}[e^{-i\tilde{\theta}}-1]\triangle,0)^{T}$,  the decomposition of space $\mathcal{X}_{\mathbb{C}}=(\mathcal{N}(T_{u}(u_{1},0)))_{\mathbb{C}}\oplus({\mathcal{X}_{2}})_{\mathbb{C}}$
and $\mathcal{N}(T_{u}(u_{1},0))=\text{span}\{{1}\}$, implies that $\psi$ takes the following form
$\psi=1+s\tilde{\beta}$ with $\tilde{\beta}\in({\mathcal{X}_{2}})_{\mathbb{C}}$. Further, we have
$\tilde{E}(\tilde{\beta},\tilde{\delta},\tilde{\theta},s)=0$, where $\tilde{E}: ({\mathcal{X}_{2}})_{\mathbb{C}}\times{\mathbb{R}}\times[0,2\pi)\times\tilde{\Upsilon}\rightarrow{\mathcal{Y}_{\mathbb{C}}}$ is denoted by
\begin{eqnarray}\label{eq70}
\tilde{E}(\tilde{\beta},\tilde{\delta},\tilde{\theta},s)=\frac{1}{s}\tilde{M}(s,\tilde{\delta},\tilde{\theta})\cdot{1}+\tilde{M}(s,\tilde{\delta},\tilde{\theta})\tilde{\beta}.
\end{eqnarray}

Next we only consider $\sigma=0$, $i.e.$, $\tilde{\theta}=0$,  since the case of $\sigma\in{\mathbb{R_{+}}}$ is still to be solved. Especially, we get $\tilde{E}(\tilde{\beta},\tilde{\delta},0,0)=\tilde{P}(\tilde{\delta},0)+T_{u}(u_{1},0)\tilde{\beta}$, where $\tilde{P}(\tilde{\delta},0)=\displaystyle\lim_{s\rightarrow 0}\frac{1}{s}\tilde{M}(s,\tilde{\delta},0)\cdot{1}$,
$i.e.$,
\begin{eqnarray}\label{eq71}
\begin{split}
\tilde{P}(\tilde{\delta},0)
=
\left(\begin{array}{c}
   u_{1}f_{u}(x,u_{1})+f(x,u_{1})-i\tilde{\delta}\\
   -r(x)g_{u}(u_{1})
\end{array}\right)
+
\left(\begin{array}{c}
   d\triangle{\eta_{1}}\\
   0
\end{array}\right)
:=\left(\begin{array}{c}
   \tilde{P}_{1}(\tilde{\delta},0)\\
   \tilde{P}_{2}(\tilde{\delta},0)
\end{array}\right).
\end{split}
\end{eqnarray}

Define $\tilde{F}: {\mathbb{R}}\times[0,2\pi)\rightarrow{\mathbb{R}}$ as $\tilde{F}(\tilde{\delta},0)=\langle{\bar{l},\tilde{P}(\tilde{\delta},0)}\rangle_{2}$. Further, we obtain
\begin{eqnarray}\label{eq73}
\begin{split}
\tilde{F}(\tilde{\delta},0)=\int_{\Omega}\tilde{P}_{1}dx-(1+du_{1})\int_{\partial{\Omega}}\tilde{P}_{2}dS=-i\tilde{\delta}\big|{\Omega}\big|+
\left \langle {\bar{l},T_{\lambda{u}}(u_{1},0)[1]} \right \rangle_{2}\neq{0},
\end{split}
\end{eqnarray}
which is a contradiction. Hence, (\ref{eq63}) has no purely imaginary eigenvalue for each $(\sigma,s)\in{\{0\}}\times\tilde{\Upsilon}$.

Now we will consider the stability of the bifurcated steady state solution $u_{01}(s)$ with $\lambda_{01}(s)$ established in Theorem \ref{th2.2} for $s\in\tilde{\Upsilon}$ as $\sigma=0$.

\begin{lemma}\label{le13}
Suppose that $(A0)$, $(A1)$, $(A2)$, $(\tilde{A3})$, $d\triangle{\eta_{1}}=0$ and $\left \langle {\bar{l},T_{uu}(u_{1},0)[1,\psi_{*}]} \right \rangle_{2}=0$ are satisfied, then for $s\in \tilde{\Upsilon}$ satisfying $s\cdot\left \langle {\bar{l},T_{\lambda{u}}(u_{1},0)[1]} \right \rangle_{2}<0$
\text{$($resp., $s\cdot\left \langle {\bar{l},T_{\lambda{u}}(u_{1},0)[1]} \right \rangle_{2}>0)$}, the eigenvalues of (\ref{eq63}) have only eigenvalues with negative real parts \text{$($resp., at least one eigenvalue}
{ with a positive real part$)$}.
\end{lemma}
\begin{proof}
From the above, we obtain that (\ref{eq63}) has no purely imaginary eigenvalue for $s\in\tilde{\Upsilon}$ as $\sigma=0$.

For each $z\in\mathbb{N}$, let $\big\{{s_{z}^{*}}\big\}_{z=1}^{\infty}\subseteq\tilde{\Upsilon}$ be a sequence, ${s_{z}^{*}}\rightarrow{0}$ as $z\rightarrow\infty$ such that
there exists an eigenvalue $\mu_{z}$ of (\ref{eq63}), and there exists a sequence $\big\{{\psi_{z}^{*}}\big\}_{z=1}^{\infty}\subseteq\mathcal{X}_{\mathbb{C}}\setminus{\{0\}}$, ${\psi_{z}^{*}}\rightarrow{1}$ as $z\rightarrow\infty$ such that
$\tilde{m}({s_{z}^{*}},\mu_{z},0){\psi_{z}^{*}}=0$, where
\begin{eqnarray}\label{eq73}
\begin{split}
   0=\triangle{\psi_{z}^{*}}+d\nabla\cdot({\psi_{z}^{*}}\nabla{u_{01}({s_{z}^{*}})})+d\nabla\cdot({u_{01}({s_{z}^{*}})}\nabla{\psi_{z}^{*}})
    +{\lambda_{01}({s_{z}^{*}})}u_{01}({s_{z}^{*}})f_{u}(x,u_{01}({s_{z}^{*}})){\psi_{z}^{*}}\\
    +{\lambda_{01}({s_{z}^{*}})}f(x,u_{01}({s_{z}^{*}})){\psi_{z}^{*}}-\mu_{z}{\psi_{z}^{*}}, ~~~~~~~~~~~~~~~~~~~~~~~~~~\text{in $\Omega$},\\
   0=\partial_{\overrightarrow{n}}{\psi_{z}^{*}}-{\lambda_{01}({s_{z}^{*}})}r(x)g_{u}(u_{01}({s_{z}^{*}})){\psi_{z}^{*}}, ~~~~~~~~~~~~~~~~~~~~~~~~~~~~~~~~~~~~~~~~~~~~~~~~~~~~\text{on $\partial{\Omega}$}.
\end{split}
\end{eqnarray}

Then it follows from $\displaystyle\lim_{z\rightarrow {\infty}}u_{01}({s_{z}^{*}})=u_{1}$ and $\langle{\bar{l},\tilde{m}({s_{z}^{*}},\mu_{z},0){\psi_{z}^{*}}}\rangle_{2}=0$ that
$\displaystyle\lim_{z\rightarrow {\infty}}\mu_{z}=0$ and
\begin{eqnarray}\label{eq74}
\begin{split}
\big|{\Omega}\big|\displaystyle\lim_{z\rightarrow {\infty}}\frac{\mu_{z}}{s_{z}^{*}}
=\left \langle {\bar{l},T_{\lambda{u}}(u_{1},0)[1]} \right \rangle_{2},
\end{split}
\end{eqnarray}
which means that $sign \Big(Re\big\{\mu_{z}\big\}\Big)=sign \Big({s_{z}^{*}}\cdot\left \langle {\bar{l},T_{\lambda{u}}(u_{1},0)[1]} \right \rangle_{2}\Big)$ for sufficiently large $z$. Hence, the proof is completed.
\end{proof}

According to the above discussion,  there are following results for the bifurcated steady state solution $u_{01}(s)$ with $\lambda_{01}(s)$ established in Theorem \ref{th2.2}.

\begin{theorem}\label{th4.2}
 Suppose that $(A0)$, $(A1)$, $(A2)$, $(\tilde{A3})$, $d\triangle{\eta_{1}}=0$ and $\left \langle {\bar{l},T_{uu}(u_{1},0)[1,\psi_{*}]} \right \rangle_{2}=0$ are satisfied.
\begin{itemize}
\item[(i)] For each $(\sigma,s)\in{\{0\}}\times\tilde{\Upsilon}$ satisfying
$s\cdot\left \langle {\bar{l},T_{\lambda{u}}(u_{1},0)[1]} \right \rangle_{2}>0$
and $\big|s\big|\ll{1}$,
(\ref{eq63}) has at least one eigenvalues with a positive real part, and the bifurcated steady state solution $u_{01}(s)$ with $\lambda_{01}(s)$ established in Theorem \ref{th2.2} is unstable.

\item[(ii)] For each $(\sigma,S)\in{\{0\}}\times\tilde{\Upsilon}$ satisfying
$s\cdot\left \langle {\bar{l},T_{\lambda{u}}(u_{1},0)[1]} \right \rangle_{2}<0$
and $\big|s\big|\ll{1}$,
    (\ref{eq63}) has only eigenvalues with negative real parts, and the bifurcated steady state solution $u_{01}(s)$ with $\lambda_{01}(s)$ established in Theorem \ref{th2.2} is linearly stable.
\end{itemize}
\end{theorem}

%

\begin{Example}\label{Example1}\
\

As for the application,  model (\ref{eq3}) can reduce to a prototypical memory-based diffusion model
with logistic growth and spatial heterogeneous resource subject to the nonlinear boundary condition
when $f(x,u)=\hat{m}(x)-u$ and  $g(u)=u^{2}$
\begin{equation}\label{eq75}
\begin{cases}
    \frac{\partial{u}}{\partial{t}}=\triangle{u}+d\nabla\cdot({u}\nabla{u_{\sigma}})+{\lambda}u(\hat{m}(x)-u),  &\text{$x \in \Omega, t >0 $},\\
     \partial_{\overrightarrow{n}}u={\lambda}r(x)u^{2},                           &\text{$x \in \partial{\Omega}, t >0 $},
  \end{cases}
\end{equation}
where $\hat{m}(x)$ is the local carrying capacity or the intrinsic growth rate that represents the situation of the resources,
the functions $f(x,u)$ and $g(u)$ satisfy those conditions in Sections 1-4.
\end{Example}

\begin{Example}\label{Example2}\
\

When we take  $f(x,u)=\hat{r}(x)\frac{\hat{k}-u}{\hat{k}+\hat{\gamma}(x)u}$ and  $g(u)=u(u-\hat{a})(1-u)$, then the model (\ref{eq3}) becomes
\begin{equation}\label{eq76}
\begin{cases}
    \frac{\partial{u}}{\partial{t}}=\triangle{u}+d\nabla\cdot({u}\nabla{u_{\sigma}})+{\lambda}u\hat{r}(x)\frac{\hat{k}-u}{\hat{k}+\hat{\gamma}(x)u},  &\text{$x \in \Omega, t >0 $},\\
     \partial_{\overrightarrow{n}}u={\lambda}r(x)u(u-\hat{a})(1-u),                           &\text{$x \in \partial{\Omega}, t >0 $},
  \end{cases}
\end{equation}
where $\hat{r}(x)$ describes the function of the growth rate of the population, $\hat{k}$ is  the constant carrying capacity of the habitat,
$\frac{\hat{r}(x)}{\hat{\gamma}(x)}$ is the replacement of mass in the population at saturation, $\hat{a}\in(0,1)$ is a constant,
the functions $f(x,u)$ and $g(u)$ satisfy those conditions in Sections 1-4.
\end{Example}

The above analysis in the previous sections could be applied to the two specific models, for brevity details are omitted here.

\section{Discussion}
In this paper, we propose a spatially heterogeneous single population model with the memory effect and nonlinear boundary condition.
Different from these previous results with the memory-based diffusion model, 
nontrivial steady state solutions are found to bifurcate  from trivial solutions $\Gamma_{0}$ and $\Gamma_{u_{1}}$, respectively,
and it is found that the stability switch and the Hopf bifurcation could be realized 
with the combined effects of the memory delay and the nonlinear boundary condition in the heterogeneous environment.
To be specific, the bifurcated steady state solution
from trivial solutions $\Gamma_{0}$, under the conditions that when the memory reaction term is stronger than the interaction of the interior reaction term and the boundary reaction term, experiences a single stability switch from stability to instability with the increase of the delayed memory value
via the Hopf bifurcation, while the bifurcated steady state solution from trivial solutions $\Gamma_{u_{1}}$, under some specific conditions, could also have the stability results
as $\sigma=0$. For the case of  $\sigma\in{\mathbb{R_{+}}}$,  it is still to be further solved. For the proposed model, there are still some interesting problems worthy of further
investigation.



\section*{Statements and Declarations}

{\bf Funding:} This work was supported by the National Natural Science Foundation of China (No. 11971032); Ji was supported by a scholarship from the China Scholarship Council
 (No. 202306500022).

\noindent{\bf Conflict of Interest}: The authors declare that they have no conflict of interest.

\noindent{\bf Data availability statement}: No data was used for the research described in this work.

\end {document}